\documentclass[11pt]{amsart}
\usepackage{amsmath,amsthm, amscd, amssymb, amsfonts,color,mathtools}
\usepackage{hhline}
\usepackage[mathscr]{eucal}
\usepackage[active]{srcltx}
\usepackage{enumitem}
\usepackage{ucs}
\usepackage{graphicx}
\usepackage{bm}
\usepackage{booktabs}
\usepackage{multirow}
\usepackage{multicol}
\usepackage{scalerel}
\usepackage{combelow}

\usepackage{color}
\usepackage[active]{srcltx}
\usepackage[all]{xy}
\usepackage{hyperref} 

\newtheorem{lema}{Lemma}[section]
\newtheorem{teor}[lema]{Theorem}
\newtheorem{cor}[lema]{Corollary}
\newtheorem{prop}[lema]{Proposition}

\theoremstyle{definition}

\newtheorem{obs}[lema]{Remark}

\newcommand{\Ss}{{\mathcal{S}}}

\newcommand{\com}{\Delta}
\newcommand{\eps}{\varepsilon}
\newcommand{\gr}{\operatorname{gr}}

\newcommand{\cS}{\mathcal{S}}
\newcommand\co{\operatorname{co}}

\newcommand\id{\operatorname{id}}

\newcommand\End{\operatorname{End}}
\newcommand\Aut{\operatorname{Aut}}

\newcommand{\Z}{{\mathbb Z}}
\newcommand{\N}{{\mathbb N}}

\newcommand{\ydgamma}{{}^{\Bbbk\Gamma}_{\Bbbk\Gamma}\mathcal{YD}}

\newcommand{\ydh}{{}^H_H\mathcal{YD}}

\newcommand\toba{{\mathfrak B }}

\newcommand\cD{\mathcal D }

\newcommand{\I}{{\mathbb I}}

\newcommand{\ad}{\operatorname{ad}_c}

\newcommand{\ord}{\operatorname{ord}}

\newcommand{\pbw}[1]{\scaleto{x}{5.5pt}_{#1}}
\newcommand{\lpbw}[1]{y_{#1}}
\newcommand{\td}[1]{\bm\tilde{#1}}
\newcommand{\cP}[3]{\mathcal{P}_{#1,#2}(#3)}
\newcommand{\mt}[1]{\mathtt{#1}}

\def\pf{\begin{proof}}
	\def\epf{\end{proof}}
\def\ot{\otimes}

\begin{document}
	\title[Liftings of Nichols algebras of type $B_3$]
	{Liftings of Nichols algebras of type $B_3$}

	\author[ D. Bagio, G. A. Garc\'ia, O. M\'arquez]
	{D. Bagio, G. A. Garc\'ia, O. M\'arquez}

	\thanks{
	The first author was partially financed by Funda\c c\~ao de Amparo a Pesquisa e 
	Inova\c c\~ao do Estado de Santa Catarina (FAPESC), Edital 21/2024. 
	The second author thankfully was partially financed by support by CAPES-PRINT, 
	Proc. 88887.895167/2023-00, CONICET and Secyt (Argentina)}
	\address{\noindent Departamento de Matem\'atica, Universidade Federal de Santa Catarina,
		88040-900, Florianópolis, SC, Brazil}
	\email{d.bagio@ufsc.br} 
	\email{oscar.marquez.sosa@ufsc.br}
	\address{\noindent G. A. G. : Departamento de Matem\'atica, Facultad de Ciencias Exactas,
		Universidad Nacional de La Plata. CONICET. Calle 47 y Calle 115, (1900)
		La Plata, Argentina.}
	\email{ggarcia@mate.unlp.edu.ar}

	\begin{abstract} 
	We give an explicit presentation of a 
	family of finite-dimensional pointed Hopf algebras over an algebraically closed field of characteristic zero 
	that constitute all liftings of Nichols algebras of diagonal Cartan type $B_{3}$ over a 
	finite abelian group.
	\end{abstract}
	
	\maketitle
	
\section{Introduction}\label{sec:introduction}

Let $\Bbbk$ be an algebraically closed field of characteristic $0$. 
The question of classifying Hopf algebras over $\Bbbk$ of a given dimension 
or with a given property goes back to the seventies. Despite of more than almost 
50 years, few general results are known. One of the most efficient methods
to study Hopf algebras with the Chevalley Property (i.e. its coradical is a 
Hopf subalgebra)
is the \emph{lifting method} 
coined by N. Andruskiewitsch and H.-J. Schneider in \cite{AS2}. 
It provides a perfect setting for describing this type of Hopf algebras through
a factorization by cosemisimple Hopf algebras and braided graded Hopf algebras.
This allows Lie theory to enter into the picture with (generalized) Dynkin diagrams, Cartan matrices and roots systems.

\smallbreak
A distinguished family of Hopf algebras with the Chevalley Property is that 
of pointed Hopf algebras, where the coradical is a group algebra. Let 
us describe in few words the lifting method in this case.
Let $A$ be a finite-dimensional Hopf algebra over a field $\Bbbk$ with coradical
$A_0=\Bbbk \Gamma$, where $\Gamma$ is a finite group. 
Then the coradical filtration is a Hopf algebra filtration 
and the corresponding graded Hopf algebra decomposes as the bosonization 
$\operatorname{gr} A\simeq R\# \Bbbk\Gamma$, where $R$ is a Hopf algebra in the
category of Yetter-Drinfeld modules $\ydgamma$. When $\Gamma$ is an abelian group,
all possible braided Hopf algebras $R$ are given by finite-dimensional 
Nichols algebras $\toba(V)$ of diagonal type. The latter are completely understood
thanks to the work of Heckenberger \cite{H2} and Angiono \cite{An}, see 
also \cite{AA} for a comprehensive reading. Therefore, to describe 
explicitly the Hopf algebra $A$ one has to 
present all possible deformations or \emph{liftings} 
of the bosonization $\toba(V)\# \Bbbk\Gamma$.

\smallbreak
A Nichols algebra $\toba(V)$
is completely determined by the braided vector space $(V,c)$
corresponding to the elements in degree one. The braided
vector space is said to be of diagonal type if 
there exists a basis $\{x_i\,:\,i\in \I_n\}$ of $V$, where $\I_n=\{1,2,\ldots,n\}$, such that 
$c(x_i\otimes x_j)=q_{ij}x_j\otimes x_i$. The braiding $c$ is 
called of {\it Cartan type} if 
\begin{equation}\label{eq:Cartan-condition}
	q_{ij}q_{ji} = q_{ii}^{a_{ij}},\quad \text{ for all }i,j\in\I_{n},
\end{equation}
for a generalized Cartan matrix
$(a_{ij})_{i,j\in \I_{n}}$. It turns out that, for all $i\in \I_{n}$,
$q_{ii}$ is a power of a unique $q_{i_{0}i_{0}}$ 
in each connected component of the Cartan 
matrix, and the latter is
a root of unity. 
When the Cartan matrix is indecomposable, one call this parameter simply 
by $q$ and its 
order by $N$.

\smallbreak
The liftings of finite-dimensional Nichols algebras of Cartan type  over a finite abelian group $\Gamma$
were completely described by Andruskiewitsch and Schneider in \cite{AS}. 
Such classification depends on a {\it data $\cD$ of 
Cartan type for $\Gamma$} together with two families of parameters which satisfy a recursion formula.
This recursion depends on the set of positive roots associated with the Cartan matrix. As the latter may be quite a bit set, an explicit description 
of the Hopf algebras giving the liftings is known only in few cases.
For ge\-ne\-ra\-lized Cartan matrices of type $A_n$, the liftings
(over a finite abelian group $\Gamma$)
were given in \cite{AS3} assuming that the parameter $q$ is a root of unity of order $N>3$. 
More recently in \cite{AAG}, the authors used a technique developed in
\cite{AAGMV}, which relies on cocycle deformations and bi-Galois objects, 
to present an explicit algorithm to compute liftings. As a main
result all liftings of finite-dimensional Nichols algebras of Cartan type $A$
over a cosemisimple Hopf algebra $H$ are classified. Helbig  
computed in \cite{Hel}, by another method and using a computer program, 
the liftings of all Nichols algebras with Cartan matrix of type $A_2$, 
some Nichols algebras with Cartan matrix of type $B_2$, and some 
Nichols algebras of two Weyl equivalence classes of non-standard type.
The explicit description for generalized Cartan matrix of type $B_n$ is 
only known when $n=2$ and $N\neq 5$; it was presented by Beattie, D\v{a}sc\v{a}lescu and Raianu
in \cite{Bea} by using 
\emph{ad-hoc} techniques. Finally, Garc\'ia Iglesias and Jury Giraldi
gave in \cite{GIJ} the liftings for Cartan type $G_{2}$ over a cosemisimple Hopf algebra by using the 
technique of cocycle deformations and bi-Galois objects. It is worth noting that 
the description of these algebras takes more than $20$ pages 
and it involves the use of a computer program.

\smallbreak
In these notes, we compute all liftings of Cartan type $B_3$
under the assumption that $N>5$ by following the approach of \cite{Bea}.
The aim of this manuscript is to have such a description at hand to compare the results obtained 
by using a new algorithm that describes the liftings as quantum subgroups of simple quantum groups 
at roots of unity. By using short exact sequences of Hopf algebras, 
the liftings would correspond to subgroups of the corresponding simple
Lie group $G$ conjugated by a unipotent element. This is a project under development 
where we obtained partial results in type $A$ and $B$.

\smallbreak
The paper is organized as follows. In Section 2, we introduce the necessary background for the rest of the paper, which includes: 
Yetter-Drinfeld modules, bosonization, Nichols algebras of diagonal (and Cartan) type and lifting of Nichols algebra. 
Proposition \ref{prop-useful} is a general result about $q$-commutation in an associative algebra, which can be useful in various contexts. 
The deformations of the power of root vector relations are computed in Section 3. The description of such deformations involves long 
but explicit calculations. 
For instance, suppose that $\alpha$ is a positive root of type $B_3$ and $y_{\alpha}^N=u_{\alpha}(\mu)$ is a root vector relation 
(see \eqref{root-vectors} for details). To determine $u_{\alpha}(\mu)$, 
one need to calculate $\Delta (y^N_{\alpha})$, which might be quite involved, especially 
when $\alpha$ has length $4$ and $5$. 
Finally, we present in Theorem \ref{lifting-b3} in Section 4, all liftings of Cartan type $B_3$ when $N>5$.


\section{Preliminaries}

		
\subsection{Notation and conventions}
We denote the natural numbers by $\N$, and $\N_0=\N\cup \{0\}$. 
For $k<\ell \in\N_0$, we set $\I_{k,\ell}=\{n \in \N_0: k \leq n \leq \ell\}$
and $\I_{\ell}=\I_{1, \ell}$.
We fix an algebraically closed field $\Bbbk$ of characteristic 0 and
we denote $\Bbbk^\times=\Bbbk-\{0\}$. 
All vector spaces and algebras are over $\Bbbk$. 
Given a Hopf algebra $H$ with  antipode  $\Ss$, the group of 
group-like elements in $H$ is denoted by $G(H)$. For $g,h\in G(H)$, 
the linear space of $(g,h)${\it -primitive  (or skew primitive)} 
elements is  
$\cP{g}{h}{H}\coloneqq \{x\in H\,:\,\Delta(x)=  g\ot x
+x\ot h \}$. 
Given $\chi\in \widehat{G(H)}$ and $g,h\in G(H)$, let $\mathcal{P}^{\chi}_{g,h}(H)=
\{x\in \cP{g}{h}{H}\,:\,lxl^{-1}=\chi(l)x,\,\text{for all } l\in G(H)\}$.
We also write $\mathcal{P}(H)=\cP{1}{1}{H}$
for the set of \emph{primitive elements}, 
$\mathcal{P}_{1}(H)$ for the set of all skew-primitive elements, $\mathcal{P}_g(H)=\cP{g}{1}{H}$ and
$\mathcal{P}^{\chi}_g(H)=\mathcal{P}^{\chi}_{g,1}(H)$.

Let  $q\in \Bbbk$.  For $n \in \N$, the $q$-numbers are defined by

\begin{align*}
	{(0)}_q   &:=   1, & &\qquad {(n)}_q  := 	{\textstyle \sum\limits_{s=0}^{n-1}} \, q^s, \qquad\qquad {(n)}_q!  := \,  {\textstyle \prod\limits_{s=0}^n} {(s)}_q \\[.3em]
	\binom{n}{k}_{\!q} &:= \frac{\,{(n)}_q!\,}{\;{(k)}_q! {(n-k)}_q!},& &\binom{n}{i_1,\ldots,i_k}_{\!q} := \frac{\,{(n)}_q!\,}{\textstyle \prod\limits_{s=1}^k {(i_s)}_q!}, \quad \text{where } i_1+\cdots +i_k=n.    
\end{align*}

\subsection{Some elementary identities}

The following general results will be used to prove some technical lemmas.

\begin{prop}\label{prop-useful}
 Let $B$ be an associative algebra, $q,\gamma_1,\gamma_2\in \Bbbk$ and $\gamma=\gamma_1\gamma_2(1-q)$. Let $z_1,\ldots,z_{2k+1}, w_{1},\ldots,w_{k} \in B$ be such that 
 \begin{align*}
z_iz_j=\begin{cases}
	qz_jz_i+w_i, & \text{ if }\,  i<j \text{ and } i+j=2k+2,\\
	qz_jz_i,& \text{ if }\, i<j \text{ and } i+j\neq 2k+2,
\end{cases}
\end{align*}
and set
$\mt{u}_k:=z_1+\ldots+ z_k+\gamma_1z_{k+1}$ and $\mt{v}_k:=\gamma_2z_{k+1}+z_{k+2}+\ldots+z_{2k+1}$. If $\sum_{i=1}^{k}w_i=-\gamma z^2_{k+1}$ then 
\[
\mt{u}_k\mt{v}_k=q\mt{v}_k\mt{u}_k.
\]
In particular, if $q$ is a primitive root of unity of order $n$ and $\gamma_1+\gamma_2=1$ then 
\[(z_1+\ldots+z_{2k+1})^n=z^n_1+\ldots+z^n_{k-1}+(\gamma_1^n+\gamma_2^n)z^n_k+z^n_{k+1}+\ldots+z_{2k+1}^n.\]
\end{prop}

\begin{proof}
It is immediate to check that
\begin{align*}
\mt{u}_k\mt{v}_k-q\mt{v}_k\mt{u}_k&=(w_1+\ldots+w_k)+\gamma_1\gamma_2(1-q)z^2_{k+1}	\\
&=-\gamma z^2_{k+1}+\gamma_1\gamma_2(1-q)z^2_{k+1}=0.
\end{align*}
Then one has that 
\begin{align*}
	(z_1+\ldots+z_{2k+1})^n&=(\mt{u}_k+\mt{v}_k)^n=\sum_{j=0}^{n}\binom{n}{j}_{q}\mt{u}^{n-j}_k\mt{v}^j_k=\mt{u}^n_k+\mt{v}^n_k\\
                           &=z^n_1+\ldots+z^n_{k-1}+(\gamma_1^n+\gamma_2^n)z^n_k+z^n_{k+1}+\ldots+z_{2k+1}^n,
\end{align*}
which finishes the proof.
\end{proof}

The following lemma follows easily by induction. 
 
\begin{lema} \label{lem-rec-grado3}
Let $B$ be an associative algebra, $\mt{x},\mt{y},\mt{z} \in B$ and  
$n\in \N$. Assume that
\begin{align*}
	\mt{x}\mt{z}=q\mt{z}\mt{x}+\mt{y},\quad \mt{y}\mt{z}=q^2\mt{z}\mt{y},\quad \mt{x}\mt{y}=q^2\mt{y}\mt{x},
\end{align*}
for some $q\in \Bbbk^{\times}$. Then
\[
	\big(\mt{z}+\mt{x}\big)^n=\sum_{i+2j+k=n} \binom{n}{i,2j,k}_{q}\left(\prod_{l=0}^{j-1}\big(2l+1\big)_{q}\right)\mt{z}^i\mt{y}^j\mt{x}^k.
	\]
\qed
\end{lema}

\begin{cor}
	 \label{cor-rec-grado3}
Let $B$, $\mt{x},\mt{y},\mt{z}$, $q$ and $n$ as in the previous Lemma. 
If $n$ is odd and $q$ is a primitive $n$-th root of unity then
\[
\big(\mt{z}+\mt{x}\big)^n=\mt{z}^n+\mt{x}^n
\]
\qed
\end{cor}


\subsection{Yetter-Drinfeld modules and bosonizations}
\label{subsec:YD-mod-boson}
In this subsection we recall the definition of 
Yetter-Drinfeld modules over Hopf algebras and
some well-known notions like 
the process called \emph{bosonization}.

\smallbreak
Let $H$ be a Hopf algebra. A \textit{left Yetter-Drinfeld module} $M$ over $H$
is a left $H$-module $(M,\cdot)$ and a left $H$-comodule 
$(M,\delta)$ with $\delta(m)=m_{(-1)}\ot m_{(0)} \in H\ot M$ 
for all $m\in M$, satisfying the compatibility 
condition 
$$ 
\delta(h\cdot m) = h_{(1)}m_{(-1)}\cS(h_{(3)})\ot h_{(2)}\cdot m_{(0)},
\qquad \text{for all }\, m\in M,\, h\in H. 
$$
Together with $H$-module and $H$-comodule maps these
modules constitute the 
category $\ydh$ of left Yetter-Drinfeld modules over $H$. 
Suppose that $H$ has bijective antipode. Then $\ydh$ is  
a braided monoidal category: for any $M, N \in \ydh$, 
the braiding $c_{M,N}:M\ot N \to N\ot M$ is given by 
$$
c_{M,N}(m\ot n) = m_{(-1)}\cdot n \ot m_{(0)},\qquad \text{for all }\, m\in M,\, n\in N.
$$

A Hopf algebra in $\ydh$ is a \emph{braided Hopf algebra}.
Let $R$ be a braided Hopf algebra with multiplication $m:R\otimes R\to R$ and braiding $c=c_{R,R}:R\otimes R\to R\otimes R$. The braided commutator between $x,y\in R$ is defined by
\[
[x,y]_c=xy-m\circ c(x\otimes y).
\]
For a primitive element $x\in R$,
the braided adjoint action of $x$ on $R$ is given by
\[
\ad(x)(y)=[x,y]_c,\quad y\in R.
\]

\bigbreak
In case $H$ is a group algebra of a 
finite abelian group, the category $\ydh$ admits a quite explicit
description.
Let $\Gamma$ be an abelian group and write $\widehat{\Gamma}$ for
the character group of $\Gamma$. Then the category $_{\Gamma}^{\Gamma}\mathcal{YD}$
of Yetter-Drinfeld modules 
over $\Bbbk\Gamma$ is just the category of left $\Bbbk[\Gamma]$-modules that are 
$\Gamma$-graded vector spaces $V= \bigoplus_{g\in \Gamma} V_{g}$ such 
that each homogeneous component $V_{g}$ is stable under the action of $\Gamma$, 
i.e. $h\cdot V_{g} \subseteq V_{g}$, for all $h,\,g \in \Gamma$.
Morphisms are simply homogeneous $\Gamma$-maps. 
The braiding $c_{V,W} : V\ot W \to W\ot V$ is given by 
$$
c_{V,W}(v\ot w) = (g\cdot w) \ot v,
\qquad \text{for all }\,  
v\in V_{g},\,w\in W.
$$

Let $V \in\ _{\Gamma}^{\Gamma}\mathcal{YD}$ be finite-dimensional
with $\dim V = n$.
In case $\Gamma$ is finite,
 the braiding $c:= c_{V,V} $ is of \emph{diagonal type} or simply 
\emph{diagonal}, i.e., there exists a basis $\{x_{1},\ldots, x_{n}\}$ such that 
$$
c(x_{i}\ot x_{j}) = q_{ij}x_{j} \ot x_{i},
$$ 
for all $ i,j\in \I_{n}$, with $q_{ij}\in \Bbbk^{\times}$.
In such a case, we say that 
$\mathbf{q}= (q_{ij})_{i,j\in \I_{n}}  $
is the  braiding matrix.
Roughly speaking, since the 
action of $\Gamma$ respects the grading,
one may decompose $V$ as 
$V =  \bigoplus_{g\in \Gamma,\chi \in \widehat{\Gamma}} V_{g}^{\chi}$,
where for any $\chi \in \widehat{\Gamma}$ we set 
$V^{\chi}=\{x\in V\, : \, h\cdot  x= \chi(h) x,\text{ for all } h\in \Gamma\}$
and $V_{g}^{\chi} = V_{g}\cap V^{\chi}$. Then 
$V$ has a basis $\{x_{i}\}_{i\in \I_{n}}$ consisting on homogeneous elements 
$x_{i} \in V_{g_{i}}^{\chi_{i}}$ for some $g_{i}\in \Gamma$ and 
$\chi_{i}\in \widehat{\Gamma}$. In particular, 
$$
c(x_{i}\ot x_{j}) = (g_{i}\cdot x_{j}) \ot x_{i} = \chi_{j}(g_{i}) x_{j} \ot x_{i},
$$
for all $i,j \in \I_{n}$. 
Then, one takes $ q_{ij} = \chi_{j}(g_{i}) $ for all $i,j \in \I_{n}$. 

\bigbreak
Let $H$ be a Hopf algebra with bijective antipode and 
$R$ a braided Hopf algebra in $\ydh$. 
The procedure to obtain a 
usual Hopf algebra from $R$ and $H$ is called 
the Majid-Radford product or \emph{bosonization}, and it is usually
denoted by $R \#H$. As vector space $R \# H = R\otimes H$, and the
multiplication and comultiplication are given by the smash-product 
and smash-coproduct,
respectively. Explicitly, for all $r, s \in R$ and $g,h \in H$,
\begin{align*}
(r \# g)(s \#h) & = r(g_{(1)}\cdot s)\# g_{(2)}h,\\ 
\com(r \# g) & =r^{(1)} \# (r^{(2)})_{(-1)}g_{(1)} \ot 
(r^{(2)})_{(0)}\# g_{(2)},
\end{align*}
where $\com_{R}(r) = r^{(1)}\ot r^{(2)}$ denotes the comultiplication in $R\in \ydh$.
If $r\in R$ and $h\in H$, then we identify $r=r\# 1$ and  
$h=1\# h$; in particular we have $rh=r\# h$ and $hr=h_{(1)}\cdot r\# h_{(2)}$.
Clearly, the map $\iota: H \to R\#H$ given by $\iota(h) = 1\#h$ for all
$h\in H$ is an injective Hopf algebra map, and the map $\pi: R\#H \to H$ 
given by $\pi(r\#h) = \eps_{R}(r)h$ for all $r\in R$, $h\in H$
is a surjective Hopf algebra map such that $\pi \circ \iota = \id_{H} $. 
Moreover, it holds that $R = (R\#H)^{\co \pi}$. 

Conversely, let $B$ be a Hopf algebra with bijective antipode and
$\pi: B\to H$ a Hopf algebra epimorphism admitting 
a Hopf algebra section $\iota: H\to B$ such that $\pi\circ\iota =\id_{H}$.
Then $R=B^{\co\pi}$ is a braided Hopf algebra in $\ydh$ and $B\simeq R\# H$
as Hopf algebras.


\subsection{Nichols algebras of diagonal type}\label{sec:Nichols-alg-diag}

Let $V$ be a vector space. A \emph{braiding} 
on $V$ is a linear map $c \in \Aut(V\ot V) $  
which is a solution to the \textit{braid equation}, that is, 
$$
(c\ot \id)(\id\ot c)(c\ot \id)=(\id\ot c)(c\ot \id)(\id\ot c) \qquad\text{ in }\End(V\ot V\ot V).
$$ 
For short, we say that $(V,c)$ is a braided vector space.
A braiding $c$ is said to be of diagonal type or simply 
diagonal if there exists a basis $\{x_{1},\ldots, x_{n}\}$ of $V$ such that $c(x_{i}\ot x_{j}) = q_{ij}x_{j} \ot x_{i}$ 
for all $ i,j\in \I_{n}$, with $q_{ij}\in \Bbbk^{\times}$.

\smallbreak
Any braided vector space of diagonal type 
can be realized as a Yetter-Drinfeld module over the group algebra of an abelian group. Indeed, 
for $(V, c)$ a finite-dimensional braided 
vector space of diagonal type,
take $\Gamma$ to be an abelian group 
containing elements  $g_{1}\ldots, g_{n}$ 
and define the characters
$\chi_{1},\ldots, \chi_{n} \in \widehat{\Gamma}$ by
$\chi_{j}(g_{i}) = q_{ij}$ for all $i, j \in\I_{n}$.
Then, $V$ is a Yetter-Drinfeld module over 
$\Bbbk \Gamma$ with $ x_{i} \in V_{g_{i}}^{\chi_{i}}$ for 
all $i\in\I_{n}$. Depending on the arithmetic properties
of the scalars $q_{ij}$ one has to choose the abelian 
group (the elements $g_{i}$ and the characters $\chi_{i}$)
accordingly to satisfy the previous equation.
For instance, the
free abelian group of rank equal to $\dim V=n$ 
with basis $g_{1}\ldots, g_{n}$ always fulfills the 
condition, so one may choose an appropriate quotient of it.

\smallbreak
Let $V$ be a Yetter-Drinfeld module over the group algebra $\Bbbk \Gamma$. 
To simplify the exposition of the notion of Nichols algebra, we assume that 
$V$ is finite-dimensional, even though this is not necessary in most places. 
The tensor algebra $T(V)$ on $V$ is a braided Hopf algebra in 
$_{\Bbbk\Gamma}^{\Bbbk\Gamma}\mathcal{YD}$ whose coproduct is 
defined by taking the elements of $V$ as 
primitive elements.
This algebra 
admits two gradings, the one on $\N_{0}$ given by the tensor 
powers and another one defined using the basis
$\{x_{1},\ldots, x_{n}\}$ of $V$.
For the latter, let $\alpha_{1},\ldots , \alpha_{n}$
be the canonical basis of $\Z^{n}$. Then $T(V)$ is a
$\Z^{n}$-graded algebra with the grading determined by
fixing $\deg(x_{i})=\alpha_{i}$ for all $i\in \I_{n}$.

\smallbreak
The \emph{Nichols algebra} $\toba(V )$ associated with a
braided vector space $(V,c)$
is the graded braided Hopf algebra given by the 
quotient of $T (V)$ by the maximal element  
$I(V )$ of the family of all
homogeneous two-sided ideals 
$I \subseteq \bigoplus_{k\geq2} T^{k} (V )$ that
are Yetter-Drinfeld submodules  and 
Hopf ideals of $T (V )$. In general, it is a hard problem to 
describe explicitly the Nichols algebra (the ideal $I(V)$ of relations) 
of a given braided vector space. 

\smallbreak
For the case of diagonal braidings one has different
characterizations of the ideal $I(V)$ and all 
finite-dimensional Nichols algebras are completely described. 
Heckenberger gave in \cite{H2} a classification 
of the braided vector spaces of diagonal type giving 
rise to finite-dimensional Nichols algebras 
in terms of generalized
Cartan matrices and generalized Dynkin diagrams.
Following \cite{AA}, one may split the 
families of these braided vector spaces
in several classes
depending on the properties of the diagonal braiding:
\emph{Standard type} (related to Lie
algebras in the Killing-Cartan classification), 
\emph{Super type} (related to the finite-dimensional contragredient Lie superalgebras in characteristic 0), 
\emph{Modular type} (related to the finite-dimensional contragredient Lie
(super)algebras in positive characteristic)
and \emph{UFO type}. 

\smallbreak
Among the braidings of standard type one has the
braidings of \emph{Cartan type}: let $(V,c)$ be a braided 
vector space of diagonal type with braiding matrix 
$\mathbf{q}= (q_{ij})_{i,j\in \I_{n}}$. The braiding
is said to be of Cartan type if there exists a generalized
Cartan matrix $\mathbf{C}=(a_{ij})_{i,j\in \I_{n}}$ satisfying \eqref{eq:Cartan-condition}.
It turns out that the Nichols algebra $\toba(V)$ associated
with such a braiding is finite-dimensional if and only 
if the Cartan matrix is of finite type
and the scalars $q_{ii}$ are (non-trivial) roots of unity 
of finite order, see \cite{H1}.
In this paper, we will be dealing with Nichols algebras
associated with diagonal braidings of Cartan type $B_{3}$ and 
$\ord q_{ii} = N >5$ for all $1\leq i\leq 3$.

	
\subsection{Nichols algebra of Cartan type $B_3$} 
\label{section-nichols-algebra}
	
Let $(V,c)$ be a braided vector space whose
braiding is diagonal of Cartan type $B_{3}$ with 
Cartan matrix and braiding matrix 
\[
\mathbf{C}=\begin{pmatrix}
	2 & -1 & 0\\
	-1 & 2 & -1\\
	0 & -2& 2.
\end{pmatrix},\qquad
\mathbf{q}
=\begin{pmatrix}
	q_{11} & q_{12} & q_{13}\\
	q_{21} & q_{22} & q_{23}\\
	q_{31} & q_{32}& q_{33}
\end{pmatrix},
\]
respectively. Write 
$\{x_1,x_2,x_3\}$ for the corresponding basis of $V$.
Since $q_{ij}q_{ji} = q_{ii}^{a_{ij}}$ 
for all $i,j\in\I_{3}$, we have that
\begin{align*}
q_{12}q_{21}&=q_{11}^{-1},& q_{21}q_{12}&=q_{23}q_{32}=q_{22}^{-1},	\\
q_{32}q_{23}&=q_{33}^{-2},& q_{13}q_{31}&=q_{31}q_{13}=1.
\end{align*}
In particular, one has $q_{11}=q_{22}=q_{33}^2$.
%
%

\smallbreak
From now on we assume $N= \ord q_{33}$ is odd. 
Then, $q_{11}$ and $q_{22}$ 
are also primitive roots of unity of order $N>1$. 
In order to present the Nichols algebra $\toba(V)$,
we need to introduce some more notation. We follow \cite{AS}
throughout these notes.

Consider the set 
of indexes $\tilde{\I}=\{1,2,3,\tilde{2},\tilde{3}\}$ 
containing $\I_{3}$. In it we define 
the following total order
%
$$
1< 2<3<\tilde{3}<\tilde{2}
$$
%

We now extend the definition of $q_{kl}$ and $a_{kl}$ 
for all $k,l \in \tilde{\I}$. 
For $j\in \tilde{\I}$, we set  
	\begin{equation}\label{eqs sin tilde kI}
		[j]=\begin{cases}
			j, & \text{ if } j\in\I_3,\\
			l,& \text{ if } j=\tilde{l}, \text{ with } l \in\I_3.
	\end{cases}\end{equation} 
Then, the extended definitions are given by:
\[
q_{kl}:=q_{[k][l]}\,\text{ and }\, a_{kl}:=a_{[k][l]}.
\]
In other words, the coefficients for the extended indices are 
the same as those for the indices  obtained by ``drooping'' the tilde,  if necessary.
In particular, note that the identities 
$q_{ij}q_{ji}=q_{ii}^{a_{ij}}$ and $q_{ij}^N=1$ remain valid for all $i,j\in \tilde{\I}_{3}$. 
Given $a,b,c,d\in \tilde{\I}$, we write
\[q_{ba,dc}=\prod_{\substack{a\leq i\leq b\\ c\leq j\leq d}} q_{ij}.\]

Let $\Phi$ denote the root system of type $B_{3}$ and write $\com =
\{\alpha_1,\alpha_2,\alpha_3\}$ for the set of simple roots 
with $\alpha_3$ the shortest one. 
Then, the other positive roots in $\Phi^{+}$ are: 
\begin{equation}
\label{nonsimple-positive-roots}	
\begin{aligned}
\alpha_{21}&=\alpha_{2}+\alpha_1,&  \alpha_{32}&=\alpha_{3}+\alpha_2,&  \alpha_{31}&=\alpha_3+\alpha_2+\alpha_1,&\\
 \alpha_{\td{3}2}&=2\alpha_{3}+\alpha_2,&\alpha_{\td{3}1}&=2\alpha_3+\alpha_2+\alpha_1,&\alpha_{\td{2}1}&=2\alpha_3+2\alpha_2+\alpha_1.&&
\end{aligned}
\end{equation}

Now, to each simple root $\alpha_{i}$ we associate the 
linear generator $x_{i}$ of $V$, so we identify 
$x_{i}= x_{\alpha_{i}}$ for all $i\in \I_{3}$. 
Also, for the non-simple positive roots 
we define the following elements in the tensor algebra $T(V)$
by using the braided commutator: 
	\begin{equation} \label{root-vector}
		\begin{aligned}
x_{\alpha_{21}} = \pbw{21}&= [x_{2},x_{1}]_{c}=
\pbw{2}\pbw{1}-q_{21}\pbw{1}\pbw{2},\\
x_{\alpha_{32}} = \pbw{32}&= 
[x_{3},x_{2}]_{c} =\pbw{3}\pbw{2}-q_{32}\pbw{2}\pbw{3},\\ 
x_{\alpha_{31}} = \pbw{31}&= 
[x_{3},x_{21}]_{c}= \pbw{3}\pbw{21}-q_{3,21}\pbw{21}\pbw{3}, \\ 
x_{\alpha_{\td{3}2}} = \pbw{\td{3}2}&=
[x_{3},x_{32}]_{c}= \pbw{3}\pbw{32}-q_{3,32}\pbw{32}\pbw{3},\\
x_{\alpha_{\td{3}1}} = \pbw{\td{3}1}&=
[x_{3},x_{31}]_{c}=\pbw{3}\pbw{31}-q_{3,31}\pbw{31}\pbw{3},\\ 
x_{\alpha_{\td{2}1}}=\pbw{\td{2}1}& = 
[x_{2},\pbw{\td{3}1}]_{c}=
			\pbw{2}\pbw{\td{3}1}-q_{2,\td{3}1}\pbw{\td{3}1}\pbw{2}.	
		\end{aligned}	
\end{equation}
These elements $x_{ab}=x_{\alpha_{ab}}$ are called 
\emph{root vectors}. The Nichols 
algebra $\toba(V)$ is presented by generators and relations 
as follows. For more details see \cite[Section 2.1]{AS}.

\begin{prop}\label{prop-nichols}
	The Nichols algebra $\toba(V)$ is generated by $\pbw{1},\pbw{2},\pbw{3}$ with defining relations 
	\begin{align}
		\label{serre-relations} \ad(\pbw{i})^{1-a_{ij}}({\pbw{j}})&=0,\qquad i\neq j,\ i,j\in \I_{3},\\[.2em]
		\label{root-vectors} \pbw{\alpha}^N&=0,\qquad   \alpha\in \Phi^{+}.
		\end{align}	
	Moreover, $\toba(V)$ has a PBW-basis given by
\[
	\{\pbw{3}^{n_9}\pbw{\td{3}2}^{n_8}\pbw{32}^{n_7}\pbw{2}^{n_6}\pbw{\td{2}1}^{n_5}\pbw{\td{3}1}^{n_4}\pbw{31}^{n_3}\pbw{21}^{n_2}\pbw{1}^{n_1}:\, 0\leq n_i< N\}.
	\]
	\qed
\end{prop}

The relations in  \eqref{serre-relations} are called \emph{quantum Serre relations} while the ones in \eqref{root-vectors} are called {\it (powers of) root vectors} relations.

\begin{obs}\label{rmk:PBW-relations}
 Set $\xi_i=(1-q_{33}^{-i})$, $i\in \N$. 
 Below we give the complete list of commutation relations 
 between root vectors. These relations 
will be used freely throughout the article without being cited.
\begin{align*}
[\pbw{32},\pbw{1}]_{c}&=\pbw{31}, &
[\pbw{\td{3}2},\pbw{2}]_{c}&=q_{32,2}\,  \xi_1 \pbw{32}^2,\\
[\pbw{\td{3}2},\pbw{1}]_{c}&=\pbw{\td{3}1}, &
[\pbw{32},\pbw{21}]_{c}&=q_{32,2}\, \xi_2 \pbw{2}\pbw{31},\\
[\pbw{32},\pbw{31}]_{c}&=q_{32,2}\xi_2\pbw{2}\pbw{\td{3}1} -q_{32,2} \pbw{\td{2}1},  &
[\pbw{\td{3}1},\pbw{21}]_{c}&= q_{31,21} \xi_1\pbw{31}^2,\\
[\pbw{\td{2}1},\pbw{1}]_{c}	&=q_{2,\td{2}1} \xi_2\pbw{\td{3}1}\pbw{21}  -q_{31,1}q_{2,32} \xi_1 \pbw{31}^2,&
[\pbw{\td{3}2},\pbw{31}]_{c}&= q_{\td{3}2,32}\xi_2\pbw{32}\pbw{\td{3}1},\\
[\pbw{\td{3}2},\pbw{21}]_{c}&= q_{32,2} \xi_2\pbw{32} \pbw{31} -(q_{\td{3}2,2})q_{3,3} \xi_1\xi_2  \pbw{2}\pbw{\td{3}1} \span +(q_{\td{3}2,2}) q_{3,3}   \pbw{\td{2}1}.
\end{align*}

The remaining relations that have not yet been specified are given by 
$[\pbw{\alpha},\pbw{\beta}]_{c}=0$.
\end{obs}


\subsection{Liftings of Nichols algebras} Let $B$ be a graded algebra. 
A \emph{lifting} or a \emph{deformation} of $B$ is a 
filtered algebra $A$ such that the associated graded algebra
$\gr A$ satisfies that $\gr A \simeq B$. 

\smallbreak
For example, let $H = \Bbbk \Gamma$ be a group algebra 
and $R$ a connected ($R^{0}= \Bbbk$) graded 
braided Hopf algebra in 
$\ydh$. Then the bosonization $B=R\# \Bbbk \Gamma$ is a graded Hopf algebra with $B^{0} = B_{0} =\Bbbk \Gamma$. 
A lifting of $B$ is a filtered
Hopf algebra $A$ such that $\gr A \simeq R\# \Bbbk \Gamma$.
In case the associated filtration of $R$ coincides with the
coradical filtration, then we say that $R$ is coradically graded. This holds for example when $R$ is a Nichols algebra.
In such a case, the filtration of $A$ is
given by the coradical filtration.

\smallbreak
Consider now a Nichols algebra $\toba(V)$ associated with a 
diagonal braiding. 
As explained in \S \ref{sec:Nichols-alg-diag}, $\toba(V)$
can be realized
as an object in the category of Yetter-Drinfeld modules over
the group algebra of an abelian group. Assume further that 
$\dim V = n$
is finite. Then, one may find a finite abelian group
$\Gamma$ together with elements  
$g_{1}\ldots, g_{n} \in \Gamma$ and characters
$\chi_{1},\ldots, \chi_{n} \in \widehat{\Gamma}$ such that
$\chi_{j}(g_{i}) = q_{ij}$ for all $i, j \in\I_{n}$.
Then, $V$ is a Yetter-Drinfeld module over 
$\Bbbk \Gamma$ with $ x_{i} \in V_{g_{i}}^{\chi_{i}}$ for 
all $i\in\I_{n}$,
$\toba(V) \in\ _{\Gamma}^{\Gamma}\mathcal{YD}$ 
and the bosonization $B= \toba(V)\#\Bbbk \Gamma$
is a usual Hopf algebra. By abuse of notation, any
lifting of $B$ is called a lifting of $\toba(V)$ over $\Gamma$, or 
simply a lifting of $\toba(V)$,
if the structure given by 
the group $\Gamma$ is somehow understood from the context.  

\smallbreak
Liftings of finite-dimensional 
Nichols algebras of Cartan type over a finite abelian group $\Gamma$
are completely described
by Andruskiewitsch and Schneider in \cite{AS}. These are classified by
\emph{Data of Cartan type for $\Gamma$}
together with two families of parameters. Such a datum is a tuple 
$\cD = \cD(\Gamma, (g_{i})_{i\in \I_{n}}, (\chi_{i})_{i\in \I_{n}}, (a_{ij})_{i,j\in \I_{n}})$ where $g_{i} \in \Gamma$, $\chi_{i} \in \widehat{\Gamma}$
and $\mathbf{C}=(a_{ij})_{i,j\in \I_{n}}$
is a Cartan matrix of finite type satisfying the Cartan 
condition \eqref{eq:Cartan-condition} given by $q_{ij}q_{ji}=q_{ii}^{a_{ij}}$
with $q_{ij} = \chi_{j}(g_{i})$ for all $i,j\in \I_{n}$; here  
the scalars $q_{ii}$ are roots of unity of order $N_{i}$. Further,
if $\mathbf{C}$ is indecomposable, then $N_{i}=N_{j}$ for all $i,j\in \I_{n}$. 
In such a case, we write $N=N_{i}$ for all $i\in \I_{n}$.
As the authors claimed, 
\emph{the explicit classification of all data of finite Cartan type
for a given abelian group is a computational problem}.

\smallbreak
Write $\Phi$ for the root system associated with $\mathbf{C}$.
In case $\mathbf{C}$ is indecomposable, only one family of parameters
is needed, and it is used to deform the power of root vectors
relations \eqref{root-vectors}. 
In a nutshell, 
the liftings are given by a family
of finite-dimensional pointed Hopf algebras $u(\cD, \mu)$
where $\cD$ is a datum of Cartan type for $\Gamma$ and 
$\mu = (\mu_{\alpha})_{\alpha \in \Phi^{+}}$ is a family of elements
in $\Bbbk$ satisfying that $\mu_{\alpha}=0$
for all $\alpha \in  \Phi^{+}$ in case $g_{\alpha}^{N}= 1$ or 
$\chi_{\alpha}^{N}\neq \varepsilon$. 

\smallbreak
Explicitly,  
$u(\cD, \mu)$ is the algebra generated by the group $\Gamma$ and
by elements $y_{\alpha}$ with $\alpha\in \Phi^{+}$ satisfying the 
relations \eqref{root-vector} and the following ones
\begin{align}
\label{comm-relations}
\text{(Action of the group)} & & gy_{i}g^{-1} &= \chi_{i}(g)y_{i} 
\qquad \text{ for all }i \in \I_{n}, g\in \Gamma ,\\[.2em]
\label{def-serre-relations}
\text{ (Serre relations)} & &
\ad(y_{i})^{1-a_{ij}}(y_{j})&=0,\qquad \text{ for all } i\neq j,\ i,j\in \I_{n},\\[.2em]
\label{root-vectors-def} 
\text{ (Root vector relations)} & &  
 y_{\alpha}^N&=u_{\alpha}(\mu),\qquad   \text{ for all } \alpha\in \Phi^{+}.
\end{align}	
The elements $u_{\alpha}(\mu)$ for $\alpha\in \Phi^{+}$ lie in the 
augmentation ideal of the Hopf subalgebra 
$\Bbbk[g_{i}^{N}\, |\, i\in \I_{n}]$. They are defined recursively in 
\cite[Section 2.1]{AS}: 
\begin{equation}\label{eqn:relation-recursion}
u_{\alpha}(\mu) = \mu_{\alpha} (1 - g_{\alpha}^{N_{J}}) + 
\sum_{\beta,\gamma\neq 0, \beta+\gamma = \alpha}
t_{\beta,\gamma}^{\alpha} \mu_{\beta}u_{\gamma}(\mu), 
\end{equation}
where $t_{\beta,\gamma}^{\alpha}\in \Bbbk$ are certain scalars determined
by the recursion. 

\bigbreak
The main problem that we solve in these notes is to 
give a closed explicit formula for each $u_{\alpha}(\mu)$. To achieve
our goal, we follow the approach of \cite{Bea} where the problem
was solved for the case of type $B_{2}$.

\bigbreak
The Hopf algebra structure of $u(\cD, \mu)$ is determined by 
$$
\com(g) = g\ot g, \qquad 
\com(y_{i}) = y_{i}\ot 1 + g_{i}\ot y_{i}\qquad\text{ for all }g\in \Gamma,\,
i\in\I_{n},
$$
Since $u(\cD, \mu)$ is generated by group-like and primitive elements,
it is a pointed Hopf algebra, and one can show that its dimension  
is $N^{|\Phi^{+}|}|\Gamma|$, for more details see \cite[Theorem 0.1]{AS}.

\bigbreak 
For the rest of the paper, we fix a braided vector space $V$ of Cartan type
$B_{3}$ with linear basis $\{x_{1},x_{2},x_{3}\}$, braiding 
matrix $\mathbf{q}=(q_{ij})_{i,j\in\I_{3}}$, 
associated Nichols algebra $\toba(V)$ as in Proposition \ref{prop-nichols}, and 
$\Gamma$ a finite abelian group with elements  
$g_{1}, g_{2}, g_{3} \in \Gamma$ and characters
$\chi_{1},\chi_{2}, \chi_{3} \in \widehat{\Gamma}$ such that
$\chi_{j}(g_{i}) = q_{ij}$ for all $i, j \in\I_{3}$.
Then, $V$ is a Yetter-Drinfeld module over 
$\Bbbk \Gamma$ with $ x_{i} \in V_{g_{i}}^{\chi_{i}}$ for 
all $i\in\I_{3}$,
$\toba(V) \in\ _{\Gamma}^{\Gamma}\mathcal{YD}$ 
and the bosonization $B= \toba(V)\#\Bbbk \Gamma$
is  a finite-dimensional pointed Hopf algebra 
with dimension $N^9|\Gamma|$. 
Also, we will denote
\begin{align}
	\label{xi-scalar} &\xi_i=(1-q_{33}^{-i})\in \Bbbk,\,\,\, i\in \mathbb{N},\\
	\label{notation-g}g_{ba}&=\prod_{a\leq j\leq b}g_j,\,\,\,\,\text{where } g_{c}=g_{[c]}\,\text{ and }\,a,b,c\in \tilde{\I}\quad(\text{see }\eqref{eqs sin tilde kI}).\\
	\label{notation-chi} \chi_{ba}&=\prod_{a\leq j\leq b}\chi_j,\,\,\,\,\text{where } \chi_{c}=\chi_{[c]}\,\text{ and }\,a,b,c\in \tilde{\I}.   \\
	\label{notation-chi-alpha} \chi_{\alpha_{ba}}&=\chi_{ba},\qquad\quad g_{\alpha_{ba}}=g_{ba}, \quad \text{for }\,\alpha_{ba}\in \Phi^{+}\setminus\Delta\quad(\text{see }\eqref{nonsimple-positive-roots}).\\[.3em]
	\label{notation-beta}\beta_1&=1/(1+q_{33})\qquad \beta_2=q_{33}/(1+q_{33}), \quad \beta=\beta_1\beta_2\xi_2.                       
\end{align}


\section{Deformation of defining relations of $\toba(V)$}
\label{section-deform}

In our goal to give a closed formula for the elements 
$u_{\alpha}(\mu)$ in \eqref{eqn:relation-recursion}, we analyse in
this section 
the possible deforming relations that must hold in any lifting of the Nichols algebra. 
Since $|\Phi^{+}| = 9$, we need to consider the deformation of 
$9$ power root vector relations coming from \eqref{root-vectors}. 
Three of them  $x_{1}^{N}$,
$x_{2}^{N}$ and $x_{3}^{N}$ are power of primitive elements, 
and correspond to simple roots.
The other six correspond to sum of simple roots and are power of elements
of higher degree. Following the definition in \eqref{root-vector}, we 
have two elements $x_{21}$, $x_{32}$ of degree two, two elements 
$x_{21}$, $x_{32}$ of degree three, 
one element $x_{\td{3}1} $ of degree four and 
one element $x_{\td{2}1}$ of degree five.

Observe that the ideal of relations $I(V)$ of $\toba(V)$ is generated by elements 
that are pri\-mi\-tive in $T(V)$ or 
are primitive modulo an ideal generated by primitive elements. 
This gives a stratification of the ideal $I(V)$ as in \cite{AAGMV}.
We analyse the possible deformations recursively using this stratification.  

\smallbreak
Fix $A$ a lifting of  $B=\toba(V)\#\Bbbk \Gamma$
and write $A_{k}$ for the $k$-th term in the coradical 
filtration of $A$;
in particular, $A_{0} = \Bbbk \Gamma$.
To simplify notation, we assume that $\gr A= B$.

\subsection{The space of skew-primitive elements in $A$}\label{subsec:primitive-elem}

We start by studying the space
of skew-primitive elements of $A$.
By \cite[Lemma 5.4 and the proof of Theorem 5.5]{AS2} 
we may choose $y_{i} \in A_{1}$ 
such that $\bar{y_{i}} = x_{i}$ in 
$B_{1}$ and $y_{i}\in P_{g_{i}}^{\chi_{i}}(A)$. 
In particular, one has that 
$\Delta(\lpbw{i})=\lpbw{i}\otimes 1+g_i\otimes \lpbw{i}$ and  $h\lpbw{i}=\chi_i(h)\lpbw{i}h$, for all $h\in \Gamma$.  
For $h=g_{j}$ this reads
$g_j\lpbw{i}=\chi_i(g_j)\lpbw{i}g_j=q_{ji}\lpbw{i}g_j$.

%
%

\smallbreak

The following lemma is \cite[Lemma 5.1]{AS3};
it will be used throughout the paper.	
	
\begin{lema}\label{lem-g-primitivos}
For all $g \in\Gamma, $ and $\chi \in \hat{\Gamma}$ with $\chi\neq \varepsilon$, we have
		\begin{enumerate}[leftmargin=*,label=\rm{(\roman*)}]
			\item  $\mathcal{P}^{\chi}_{g}(A)\neq 0$ if and only if there exists $k\in \I_{3}$ such that $g=g_k$ and $\chi = \chi_k$,\vspace{.1cm}
			\item $\mathcal{P}^{\varepsilon}_{g}(A)=\Bbbk(1-g)$. \qed
		\end{enumerate}
\end{lema}

\smallbreak
The next result is another version of 
\cite[Lemma 2.1]{Bea}. We give its proof for completeness.

\begin{lema} \label{lem-for-lifting}
	The following assertions hold
	\begin{enumerate}[leftmargin=*,label=\rm{(\roman*)}]
		\item If $\chi^N_i\neq \varepsilon$ then $\mathcal{P}^{\chi^N_i}_{g^N_i}(A)=0$, for all $i=1,2,3$. \vspace{.1cm}
		\item For $\alpha\in \Phi^{+}\setminus\Delta$, either $\chi_{\alpha}^N=\varepsilon$ or $\mathcal{P}^{\chi_{\alpha}^N}_{g^N_{\alpha}}(A)=0$.\vspace{.1cm}
		\item Let $i,j\in \I_3$ with $i\neq j$. If $\chi_j\chi_{i}^{1-a_{ij}}=\epsilon$ then $N\in \{3,5\}$.
		\item Let $i,j\in \I_3$ with $i\neq j$. If $N\neq 3,5$ then $\mathcal{P}^{\chi_j\chi_{i}^{1-a_{ij}}}_{g_{j}g_{i}^{1-a_{ij}}}(A)=0$.
	\end{enumerate}
\end{lema}

\begin{proof}
(i) Suppose that $\mathcal{P}^{\chi^N_i}_{g^N_i}(A)\neq 0$. By Lemma \ref{lem-g-primitivos}, 
there exist $j\in \I_3$	such that $\chi^N_i=\chi_j$ and $g^N_i=g_j$. 
Therefore,
$1=q_{jj}^{a_{ji}N}=q_{ji}^Nq_{ij}^N=\chi_{i}^N(g_j)\chi_{j}(g_i^N)=\chi_{j}(g_j)\chi_{j}(g_j)=q_{jj}^2$.
This contradicts the fact that $q_{jj}$ has order $N$, which is odd.
	
(ii) Let $\alpha=\alpha_{ji}\in \Phi^{+}\setminus\Delta$	
and suppose that $\chi_{ji}^N\neq\varepsilon$ and $\mathcal{P}^{\chi_{ji}^N}_{g_{ji}^N}(A)\neq0$.
 Then, again by Lemma \ref{lem-g-primitivos}, there exists 
 $k\in \I_3$ such that $\chi_{ji}^N=\chi_k$ and $g_{ji}^N=g_k$. 
 Therefore,
$q_{k,ji}^N=\chi_{ji}^N(g_k)=\chi_k(g_k)=q_{kk}$, $q_{ji,k}^N=\chi_k(g_{ji}^N)=\chi_k(g_k)=q_{kk}$
and we have that $q_{k,ji}^Nq_{ji,k}^N=q_{kk}^{2}$.	
Since  
\[
q_{k,ji}^Nq_{ji,k}^N=\prod_{i\le l \le j} q_{kl}^N q_{lk}^N=\prod_{i\le l \le j} q_{kk}^{a_{kl}N}=1,
\]
it follows that $q_{kk}^{2}=1$, which is a contradiction.
	
(iii) If $\chi_j\chi_{i}^{1-a_{ij}}=\epsilon$ then $q_{ij}q_{ii}^{1-a_{ij}}=1$ and $q_{jj}q_{ji}^{1-a_{ij}}=1$. Thus,
	\[1= q_{jj}  q_{ji}^{1-a_{ij}}  q_{ij}^{1-a_{ij}}  q_{ii}^{(1-a_{ij})^2} =q_{jj} q_{ii}^{a_{ij}(1-a_{ij})}  q_{ii}^{(1-a_{ij})^2} =  q_{jj}q_{ii}^{(1-a_{ij})}.   \]
	Using that $q_{ll}=q_{33}^{2-\delta_{l,3}}$, for all $l\in \I_{3}$,  we have $1=q_{33}^{(2-\delta_{j,3})+(2-\delta_{i,3})(1-a_{ij}) }$.
	Since $1\le (2-\delta_{i,3})(1-a_{ij}) + (2-\delta_{j,3})\le 6$ and $N$ is odd, it follows that $N\in\{3,5\}$.
	
(iv)  Suppose that $N\neq 3,5$ and $\mathcal{P}^{\chi_j\chi_{i}^{1-a_{ij}}}_{g_{j}g_{i}^{1-a_{ij}}}(A)\neq 0$. 
	From (iii), $\chi_j\chi_{i}^{1-a_{ij}}\neq\epsilon$, and consequently   
	Lemma \ref{lem-g-primitivos} implies that there exist $k\in \I_3$ such that 
	$\chi_j\chi_{i}^{1-a_{ij}}=\chi_k$, and $g_{j}g_{i}^{1-a_{ij}}=g_k$.
	Hence, $q_{ii}^{1-a_{ij}}q_{ij}=q_{ik}$, $q_{ii}^{1-a_{ij}}q_{ji}=q_{ki}$ and we obtain
	\[q_{ii}^{2-a_{ij}} =q_{ii}^{2(1-a_{ij})}q_{ii}^{a_{ij}}   =q_{ii}^{2(1-a_{ij})}q_{ij}q_{ji}=q_{ik}q_{ki}=q_{ii}^{a_{ik}}.\]
	Thus, $q_{ii}^{2-a_{ij}-a_{ik}}=1$. Since $0\le 2-a_{ij}-a_{ik}\le 6$ and $N\ge 7$ we have a contradiction.
\end{proof}

\begin{obs}\label{obs-repe-veces}
In what follows we use repeatedly the following argument:
For $j,i\in \tilde{\I}$ let $y_{ji}  \in A$ be 
such that $y_{ji}^N=0$ if $g_{ji}^N=1$ or $\chi_{ji}^N\neq\varepsilon$.  
Then 
\[
q_{k,ji}^N\,  y_{ji}^N=y_{ji}^N,\quad \text{for all }k\in \tilde{\I}.
\]
Indeed, since $\chi_{ji}^N(g_k)=q_{k,ji}^N$, 
$\chi_{k}(g_{ji}^{N}) =q_{ji,k}^N$
and $q_{k,ji}^{N}q_{ji,k}^{N} = q_{kk}^{N} = 1$, the equality above 
clearly holds if $g_{ji}^N=1$ or $q_{k,ji}^N=1$. Assume otherwise that $q_{k,ji}^N\neq 1$, then $\chi_{ji}^N(g_k)=q_{k,ij}^N\neq 1$ and whence $\chi_{ji}^N\neq\varepsilon$.
This implies that $y_{ji}^N=0$ and the equality also holds.

Taking $k=j=i+1$ and using that 
$q_{(i+1)i}^N=q_{i+1,(i+1)i}^N$, we also have that 
\[
q_{(i+1)i}^N y_{(i+1)i}^N=y_{(i+1)i}^N.
\]
Similarly, for each $y_{i} \in A$ such that $y_{i}^N=0$ if $g_{i}^N=1$ or $\chi_{i}^N\neq\varepsilon$ we have 
\[q_{k,i}^N y_{i}^N=y_{i}^N,\quad \text{for all }k\in \tilde{\I}.\]
\end{obs}

\subsection{Deformation of $N$-th powers of root vectors of degree one}
\label{subsection-grade1}

The next result gives us the possible deformations of the relations given by powers of
primitive elements in $\toba(V)$ of degree one, which
are $x_{1}, x_{2}, x_{3}$.

As in the beginning of \S \ref{subsec:primitive-elem},
we fix $y_{i} \in A_{1}$ 
such that $\bar{y_{i}} = x_{i}$ in 
$B_{1}$ and $y_{i}\in P_{g_{i}}^{\chi_{i}}$. 
In particular, $\bar{y_{i}}^{N} = x_{i}^{N}=0$ for all $1\leq i\leq 3$.

\begin{lema} \label{lem:def-degree-1} For each $1\leq i\leq 3$,
there exists $\mu_{i}\in \Bbbk$ such that 
\begin{align}\label{def-rel-grade-1}
	\lpbw{i}^N=\mu_i(g_i^N-1),
\end{align}	
and $\mu_i=0$ if $g_i^N=1$ or $\chi^N_i\neq \varepsilon$, and $\mu_i=1$ if $\chi^N_i=\varepsilon$.
\end{lema}

\begin{proof}
Since $\lpbw{i}\in \mathcal{P}^{\chi_i}_{g_i}(A)$ for all $i\in \I_3$, 
by a straightforward computation one gets 
\[
\Delta(\lpbw{i}^k)=\sum_{j=0}^{k}\binom{k}{j}_{q_{ii}}\lpbw{i}^{k-j}g^j_i\otimes \lpbw{i}^j, \qquad k\in \mathbb{N}.
\]
Using that $q_{ii}$ is a primitive $N$-th root of unity, it follows that
$\Delta(\lpbw{i}^N)=\lpbw{i}^N\otimes 1+g^N_i\otimes \lpbw{i}^N$, which
implies that 
$\lpbw{i}^N\in \mathcal{P}^{\chi^N_i}_{g_i^N}(A)$. 
If $\chi^N_i\neq \varepsilon$ then $\lpbw{i}^N=0$,
by Lemma \ref{lem-for-lifting} (i).
On the other hand, if  $\chi^N_i= \varepsilon$ then 
$\lpbw{i}^N\in 	\Bbbk (g^N_i-1)$. Hence, there exists 
$\mu_i\in \Bbbk$ such that $\lpbw{i}^N=\mu_i(g^N_i-1)$. 
By a change on variable on $y_i$, if necessary, we have that
$\lpbw{i}^N=\mu_i(g_i^N-1)$, where $\mu_i=0$ if $g_i^N=1$ or $\chi^n_i\neq \varepsilon$, and $\mu_i=1$ if $\chi^n_i=\varepsilon$.
\end{proof}

\subsection{Deformation of the $N$-th power of root vectors of degree 2}\label{subsection-grade2}
In this subsection 
we compute the possible deformations of the relations given by powers of
root vectors of degree two. These elements 
are given by $\pbw{21}=\pbw{2}\pbw{1}-q_{21}\pbw{1}\pbw{2}$ and 
$\pbw{32}=\pbw{3}\pbw{2}-q_{32}\pbw{2}\pbw{3}$. 

For $i=1,2$, set 
$\lpbw{(i+1)i}=\lpbw{i+1}\lpbw{i}-q_{(i+1)i}\lpbw{i}\lpbw{i+1}\in A$. 

\begin{lema} \label{lem:def-degree-2} 
For each $1\leq i\leq 2$, there exist scalars $\mu_{(i+1)i}\in \Bbbk$ 
such that
\begin{align}\label{def-rel-grade-2}
		\lpbw{(i+1)i}^N=\mu_{(i+1)i}\big(g_{(i+1)i}^N-1\big)-\xi_2^N\mu_{i+1}\lpbw{i}^N,
\end{align}
and
$\mu_{(i+1)i}=0$ if $g_{(i+1)i}^N=1$ or $\chi_{(i+1)i}^N\neq\varepsilon$.
\end{lema}

\begin{proof}
A direct computation shows that
\[
\Delta(\lpbw{(i+1)i})=
\lpbw{(i+1)i} \otimes 1 +  \xi_2 \lpbw{i+1}\,  g_{i} \otimes \lpbw{i} + 
g_{(i+1)i}\otimes \lpbw{(i+1)i}.
\]
Take $i=1$ and write  
$\Delta(\lpbw{21})=\mt{a}_1+\mt{a}_2+\mt{a}_3$ with 
$\mt{a}_1=\lpbw{21} \otimes 1$, $\mt{a}_2=\xi_2 \lpbw{2}\,  g_{1} \otimes \lpbw{1}$ and $\mt{a}_3=g_{21}\otimes \lpbw{21}$. 
Since $\mt{a}_k\mt{a}_j=q^{-2}_{33}\mt{a}_j\mt{a}_k$ 
for all $1\leq j<k\leq 3$, it follows that
$$
\Delta(\lpbw{21}^N)=\mt{a}^N_1+\mt{a}^N_2+\mt{a}^N_3
=\lpbw{21}^N\otimes 1+q^{\frac{N(N-1)}{2}}_{12}\xi_2^N\lpbw{2}^Ng_1^N\otimes \lpbw{1}^N
+g_{21}^N\otimes \lpbw{21}^N.
$$
Similarly, for $i=2$ we write
$\Delta(\lpbw{32})=\mt{b}_1+\mt{b}_2+\mt{b}_3$ with 
$\mt{b}_1=\lpbw{32} \otimes 1$, 
$\mt{b}_2=\xi_2 \lpbw{3}\,  g_{2} \otimes \lpbw{2}$ and 
$\mt{b}_3=g_{32}\otimes \lpbw{32}$. 
In this case,
\begin{align*}
	\mt{b}_1\mt{b}_2&=q^{-1}_{33}\mt{b}_2\mt{b}_1+\mt{b}_4,& \mt{b}_j\mt{b}_3&=q^{-1}_{33}\mt{b}_3\mt{b}_j,\,\,\, j=1,2, \\
	\mt{b}_1\mt{b}_4&=q^{-2}_{33}\mt{b}_4\mt{b}_1, & \mt{b}_4\mt{b}_2&=q^{-2}_{33}\mt{b}_2\mt{b}_4,&   
\end{align*}
where $\mt{b}_4=-\xi_2q^{-1}_{32}q^{-1}_{33}\lpbw{\td{3}2}g_2\otimes \lpbw{2}$. Hence
$$	
\Delta(\lpbw{32}^N)=\mt{b}^N_1+\mt{b}^N_2+\mt{b}^N_3
=\lpbw{32}^N\otimes 1+q^{\frac{N(N-1)}{2}}_{23}\xi_2^N\lpbw{3}^Ng_2^N\otimes \lpbw{2}^N+
g_{32}^N\otimes \lpbw{32}^N.
$$
If $q^N_{i(i+1)}\neq 1$ then $\chi^N_{i+1}\neq \varepsilon$ and  
by Remark \ref{obs-repe-veces} we get 
$q^{\frac{N(N-1)}{2}}_{i(i+1)}y_{i+1}^{N}=y_{i+1}^{N}$
and 
\begin{align*}
	\Delta(\lpbw{(i+1)i}^N)&=\lpbw{(i+1)i}^N\otimes 1+\xi_2^N\lpbw{i+1}^Ng_i^N\otimes \lpbw{i}^N+g_{(i+1)i}^N\otimes \lpbw{(i+1)i}^N\\[.2em]
	                       &\overset{\mathclap{\eqref{def-rel-grade-1}}}{=}\lpbw{(i+1)i}^N\otimes 1+\xi_2^N\mu_{i+1}(g_{i+1}^N-1)g_i^N\otimes \lpbw{i}^N+g_{(i+1)i}^N\otimes \lpbw{(i+1)i}^N\\[.3em]
	                       &=\lpbw{(i+1)i}^N\otimes 1+\xi_2^N\mu_{i+1}(g_{i+1}^N-1)g_i^N\otimes \lpbw{i}^N+g_{(i+1)i}^N\otimes \lpbw{(i+1)i}^N.
\end{align*}
Therefore, 
$\nu_{(i+1)i}:=\lpbw{(i+1)i}^N+\xi_2^N\mu_{i+1}\lpbw{i}^N\in \mathcal{P}_{g^N_{(i+1)i}}^{\chi_{(i+1)i}^N}$. Now,
by Lemma \ref{lem-for-lifting} (ii), there exists  $\mu_{(i+1)i}\in \Bbbk$ such that $\nu_{(i+1)i}=\mu_{(i+1)i}(g_{(i+1)i}^N-1)$. Consequently, we have that 
\[\lpbw{(i+1)i}^N=\mu_{(i+1)i}\big(g_{(i+1)i}^N-1\big)-\xi_2^N\mu_{i+1}\lpbw{i}^N,\]	
where $\mu_{(i+1)i}=0$ if $g_{(i+1)i}^N=1$ or \nolinebreak $\chi_{(i+1)i}^N\neq\varepsilon$. 
\end{proof}

\subsection{Deformation of the $N$-th power of root vectors of degree 3}\label{grade-3}
Below we determine the possible 
deformations of the relations given by powers of
root vectors of degree three, which are  
$\pbw{31} =  \pbw{3}\pbw{21}-q_{3,21}\pbw{21}\pbw{3}$
and  $\pbw{\td{3}2}=\pbw{3}\pbw{32}-q_{3,32}\pbw{32}\pbw{3}$.

Define the following elements of $A$:
\begin{equation}\label{def-elements}
\begin{aligned}
\lpbw{31}&:=\lpbw{3}\lpbw{21}-q_{3,21}\lpbw{21}\lpbw{3},& 
\lpbw{\td{3}2}&=\lpbw{3}\lpbw{32}-q_{3,32}\lpbw{32}\lpbw{3},\\
\lpbw{\td{3}1}&:=\lpbw{3}\lpbw{31}-q_{3,31}\lpbw{31}\lpbw{3},&
\lpbw{\td{2}1}&=\lpbw{2}\lpbw{\td{3}1}-q_{2,\td{3}1}\lpbw{\td{3}1}\pbw{2}.
\end{aligned}
\end{equation}

\begin{lema} \label{lem:def-degree-3} 
There exist scalars $\mu_{31}$, $\mu_{\td{3}2} \in \Bbbk$ such 
that the following identities hold in $A$:
	\begin{align}
\label{def-rel-grade-3=y321} &\lpbw{31}^N=\mu_{31}(g_{31}^N-1)-
\xi_2^N\mu_{3}\lpbw{21}^N-\xi_2^N\mu_{32}\lpbw{1}^N, \\[.2em]	
\label{def-rel-grade-3=y332} &\lpbw{\td{3}2}^N=\mu_{\td{3}2}(g_{\td{3}2}^N-1)-
2\xi_1^N\mu_{3}\lpbw{32}^N-\xi_1^N\xi_2^N\mu_{3}^2\lpbw{2}^N,	
\end{align}
where $\mu_{31}=0$ if $g_{31}^N=1$ or 
$\chi_{31}^N\neq\varepsilon$ and $\mu_{\td{3}2}=0$ if 
$g_{\td{3}2}^N=1$ or $\chi_{\td{3}2}^N\neq \varepsilon$.
\end{lema}

\begin{proof}
Write
$\Delta(\lpbw{31})=\mt{a}_1+\mt{a}_2+\mt{a}_3+\mt{a}_4$ with
$$
	\mt{a}_1=\lpbw{31} \otimes 1, \qquad	
	\mt{a}_3= \xi_2 \lpbw{3} g_{21} \otimes \lpbw{21}, \qquad
	\mt{a}_2= \xi_2 \lpbw{32} g_1 \otimes \lpbw{1}, \qquad 
	\mt{a}_{4}= g_{31}\otimes \lpbw{31}. 
$$

\noindent {\it Claim 1}: $\Delta(\lpbw{31}^N)=\mt{a}^N_1+
\mt{a}^N_2+\mt{a}^N_3+\mt{a}^N_4$. 

\smallbreak

\noindent Firstly, notice that 
$(\mt{a}_1+\mt{a}_2+\mt{a}_3)\mt{a_4}= q^{-1}_{33}\mt{a_4}(\mt{a}_1+\mt{a}_2+\mt{a}_3)$. 
Since $q^{N}_{33}=1$, we obtain that 
\[
\Delta(\lpbw{31}^N)=(\mt{a}_1+\mt{a}_2+\mt{a}_3)^N+\mt{a}^N_4.
\]
Now, we analyse the term $(\mt{a}_1+\mt{a}_2+\mt{a}_3)^N$. For this, consider 
\[ 
\mt{a}_5=-\xi_2^2  q_{33}^{-1} q_{32,21}^{-1}   \pbw{\td{3}2}g_{21}g_1 \otimes \lpbw{21}\lpbw{1}, 
\qquad \mt{a}_6= \xi_2    q_{33}^{-1} q_{3,21}^{-1}    \lpbw{\td{3}1}g_{21} \otimes\lpbw{21}.
\]
Then, taking $\mt{a}_{12}=\mt{a}_1+\mt{a_2}$ and $\mt{a}_{56}=\mt{a}_5+\mt{a_6}$,
one has that
\[
\mt{a}_{12}\mt{a}_3=q^{-1}_{33}\mt{a}_3\mt{a}_{12}+\mt{a}_{56},\quad \mt{a}_{56}\mt{a}_3
=q_{33}^{-2}\mt{a}_3\mt{a}_{56},\quad \mt{a}_{12}\mt{a}_{56}=q_{33}^{-2}\mt{a}_{56}\mt{a}_{12}.
\]
Applying Corollary \ref{cor-rec-grado3} for $\mt{x}=\mt{a}_{12}$, $\mt{y}=\mt{a}_{56}$
 and $\mt{z}=\mt{a}_3$ we conclude that
\[
(\mt{a}_1+\mt{a}_2+\mt{a}_3)^N=(\mt{a}_{12}+\mt{a}_3)^N=
\mt{a}^N_{12}+\mt{a}^N_3.
\]
Similarly, write $\mt{a}_{78}=\mt{a}_7+\mt{a}_8$ where
\[
\mt{a}_7= -\xi_2^2 (q_{32}q_{33})(q_{21}q_{31})^{-1}  \lpbw{2}\lpbw{\td{3}1}\, 
g_1 \otimes\lpbw{1},\qquad \mt{a}_8=
 \xi_2 (q_{32}q_{33})(q_{21}q_{31})^{-1} \lpbw{\td{2}1} g_1 \otimes \lpbw{1}.
\]
Then
\[
\mt{a}_{1}\mt{a}_2=q^{-1}_{33}\mt{a}_2\mt{a}_{1}+\mt{a}_{78},\quad
 \mt{a}_1\mt{a}_{78}=q_{33}^{-2}\mt{a}_{78}\mt{a}_1,\quad 
 \mt{a}_{78}\mt{a}_{2}=q_{33}^{-2}\mt{a}_{2}\mt{a}_{78}.
\]
Applying again Corollary \ref{cor-rec-grado3} for 
$\mt{x}=\mt{a}_{1}$, $\mt{y}=\mt{a}_{78}$ and $\mt{z}=\mt{a}_2$
it follows that $(\mt{a}_1+\mt{a}_2)^N=\mt{a}^N_1+\mt{a}^N_2$ 
and the claim is proved. 

\smallbreak
Set $c_1=(q_{1,32})^{\frac{N(N-1)}{2}}$ and $c_2=(q_{21,3})^{\frac{N(N-1)}{2}}$.
From Claim 1 and Remark \ref{obs-repe-veces} it follows 
\begin{align*}
	\Delta(\lpbw{31}^N)&=\lpbw{31}^N\otimes 1+c_1\lpbw{32}^N g^N_1 \otimes \lpbw{1}^N+
	c_2\lpbw{3}^N g^N_{21} \otimes \lpbw{21}^N+g^N_{31}\otimes \lpbw{31}^N\\[.2em]
	&\overset{\mathclap{\eqref{def-rel-grade-2}}}{=}\lpbw{31}^N\otimes 1+
	c_1\big(\mu_{32}\big(g_{32}^N-1\big)-\xi_2^N\mu_{3}\lpbw{2}^N\big) g^N_1 \otimes \lpbw{1}^N+
	c_2\lpbw{3}^N g^N_{21} \otimes \lpbw{21}^N+g^N_{31}\otimes \lpbw{31}^N\\[.2em]
	&\overset{\mathclap{\eqref{def-rel-grade-1}}}{=}\lpbw{31}^N\otimes 1+c_1\big(\mu_{32}\big(g_{32}^N-1\big)-\xi_2^N\mu_{3}\mu_2(g_2^N-1)\big) g^N_1 \otimes \lpbw{1}^N+\\[.2em]
	&\quad+c_2\mu_3(g_3^N-1) g^N_{21} \otimes \lpbw{21}^N+g^N_{31}\otimes \lpbw{31}^N\\[.2em]
	&=\lpbw{31}^N\otimes 1+\big(\mu_{32}\big(g_{32}^N-1\big)-
	\xi_2^N\mu_{3}\mu_2(g_2^N-1)\big) g^N_1 \otimes \lpbw{1}^N+\\[.2em]
	&\quad+\mu_3(g_3^N-1) g^N_{21} \otimes \lpbw{21}^N+g^N_{31}\otimes \lpbw{31}^N.
\end{align*}
Thus
$\nu_{31}=\lpbw{31}^N+\xi_2^N\mu_{3}\lpbw{21}^N+\
xi_2^N\mu_{32}\lpbw{1}^N \in \mathcal{P}_{g^N_{31}}^{\chi_{31}^N}$, and
by Lemma \ref{lem-for-lifting} (ii), there exists $\mu_{31}\in \Bbbk$ 
such that $\nu_{31}=\mu_{31}(g_{31}^N-1)$. Then
\[
	\lpbw{31}^N=\mu_{31}(g_{31}^N-1)-\xi_2^N\mu_{3}\lpbw{21}^N-
	\xi_2^N\mu_{32}\lpbw{1}^N,	
\]
where $\mu_{31}=0$ if $g_{31}^N=1$ or $\chi_{31}^N\neq\varepsilon$.
\vspace{.2cm}
Now, we compute $\Delta(\lpbw{\td{3}2}^N)$. 
As before, write $\Delta(\lpbw{\td{3}2})=\mt{b}_1+\mt{b}_2+\mt{b}_3+\mt{b}_4$, 
where
$$
\mt{b}_1=\lpbw{\td{3}2}\otimes 1, \qquad \mt{b}_3=\xi_1 \xi_2 \lpbw{3}^2 g_2\otimes \lpbw{2},  \qquad
\mt{b}_2=q_{3,3}\xi_2\lpbw{3} g_{32}\otimes \lpbw{32}, \qquad 
\mt{b}_{4}=g_{\td{3}2}\otimes \lpbw{\td{3}2}.
$$ 

\smallbreak
\noindent {\it Claim 2: $\Delta(\lpbw{\td{3}2}^N)=\mt{b}_1^N+
\mt{b}_2^N+2(1+q_{33})^{-2}\mt{b}_3^N+\mt{b}_4^N+\mt{b}_5^N$}. 

\smallbreak
\noindent Since $\mt{b}_1(\mt{b}_2+\mt{b}_3+\mt{b}_4)=
q^{-2}_{33}(\mt{b}_2+\mt{b}_3+\mt{b}_4)\mt{b}_1$
and $q^{N}_{33}=1$, it follows that
$\Delta(\lpbw{\td{3}2}^N)=\mt{b}^N_1+(\mt{b}_2+\mt{b}_3+\mt{b}_4)^N$.
Now observe that
\begin{align*}
&\mt{b}_2\mt{b}_3=q_{33}^{-2}\mt{b}_3\mt{b}_2,& 
&\mt{b}_2\mt{b}_4=q_{33}^{-2}\mt{b}_4\mt{b}_2-\beta\mt{b}^2_3,&		
&\mt{b}_3\mt{b}_4=q_{33}^{-2}\mt{b}_4\mt{b}_3,&	
\end{align*}
where $\beta$ is given in \eqref{notation-beta}. 
Using Proposition \ref{prop-useful} for $z_i=\mt{b}_{i+1}$, $i\in \I_3$, 
and $w_1=-\beta\mt{b}^2_3$, we obtain 
\[
(\mt{b}_2+\mt{b}_3+\mt{b}_4)^N=(\mt{b}_2+\beta_1\mt{b}_3)^N
+(\beta_2\mt{b}_3+\mt{b}_4)^N=\mt{b}^N_2+(\beta_1^N+
\beta^N_2)\mt{b}^N_3+\mt{b}^N_4, 
\]
which proves the claim because 
$\beta_1^N+\beta^N_2=2(1+q_{33})^{-2}$. \vspace{.3cm}

Set $d_1=(q_{32,3})^{\frac{N(N-1)}{2}}$ and $d_2=q_{23}^{N(N-1)}$. Then,
from Claim $2$ and Remark \ref{obs-repe-veces}, we have 
\begin{align*}
\Delta(\lpbw{\td{3}2}^N)&=\lpbw{\td{3}2}^N\otimes 1	+
d_1\xi_2^N\lpbw{3}^Ng_{32}^N\otimes \lpbw{32}^N+
2d_2(1+q_{33})^{-N}\xi_1^N\xi_2^N\lpbw{3}^{2N}g_2^N\otimes y_2^N+
g_{\td{3}2}^N\otimes \lpbw{\td{3}2}^N\\[.2em]
                        &=\lpbw{\td{3}2}^N\otimes 1	
                        +\xi_2^N\mu_3(g_3^N-1)g_{32}^N\otimes \lpbw{32}^N+ \\
                        &\quad + 2(1+q_{33})^{-N}\xi_1^N\xi_2^Nq_{23}^{N(N-1)} 
                        \mu_3^{2N}(g_3^N-1)^{2N}g_2^N\otimes y_2^N+g_{\td{3}2}^N\otimes \lpbw{\td{3}2}^N.                    
\end{align*}
Hence
$\nu_{\td{3}2}=\lpbw{\td{3}2}^N+2\xi_1^N\mu_{3}\lpbw{32}^N+
\xi_1^N\xi_2^N\mu_{3}^2\lpbw{2}^N\in \mathcal{P}_{g^N_{\td{3}2}}^{\chi_{\td{3}2}^N}$,
and by Lemma \ref{lem-for-lifting} (ii), 
there exists $\mu_{\td{3}2}\in \Bbbk$ such that $\nu_{\td{3}2}=\mu_{\td{3}2}(g_{\td{3}2}^N-1)$. 
Consequently
$\lpbw{\td{3}2}^N=\mu_{\td{3}2}(g_{\td{3}2}^N-1)-2\xi_1^N\mu_{3}\lpbw{32}^N
-\xi_1^N\xi_2^N\mu_{3}^2\lpbw{2}^N$,	
where $\mu_{\td{3}2}=0$ if $g_{\td{3}2}^N=1$ or $\chi_{\td{3}2}^N\neq\varepsilon$.
\end{proof}

\subsection{Deformation of $N$-th power of the root vector of degree 4}\label{grade-4}
Below we give the possible 
deformation of the relation given by the $N$-th power of 
the root vector of degree four, which is given by 
the element $\pbw{\alpha_{\td{3}1}}=  \pbw{3}\pbw{31}-q_{3,31}\pbw{31}\pbw{3}$. 

\begin{lema}\label{lem-grade-4} 
There exists a scalar $\mu_{\td{3}1}\in\Bbbk$ such that the following identity holds in $A$:
\begin{align}\label{def-rel-grade-3=y3321}
\lpbw{\td{3}1}^N=\mu_{\td{3}1}(g_{\td{3}1}^N-1)-2\xi_1^N\mu_{3}\lpbw{31}^N-
\xi_1^N\xi_2^N\mu_{3}^2\lpbw{21}^N-\xi_2^N\mu_{\td{3}2}\lpbw{1}^N,
\end{align}
with $\mu_{\td{3}1}=0$ if $g_{\td{3}1}^N=1$ or 
$\chi_{\td{3}1}^N\neq\varepsilon$.
\end{lema}

\begin{proof}
As in the proofs above, we decompose the coproduct of $\lpbw{\td{3}1}$
as $\Delta(\lpbw{\td{3}1})= \mt{a}_1+\mt{a}_2+\mt{a}_3+\mt{a}_4+\mt{a}_5$, 
where
\begin{align*}
	\mt{a}_1&=\lpbw{\td{3}1}\otimes 1, & \mt{a}_2&= 
	\xi_2\lpbw{\td{3}2}g_1\otimes \lpbw{1}, & \mt{a}_3&= 
	\xi_2\xi_1\lpbw{3}^2g_{21}\otimes\lpbw{21},\\
\mt{a}_4&=q_{3,3} \xi_2 \lpbw{3}g_{31}\otimes \lpbw{31},&	
\mt{a}_5&= g_{\td{3}1}\otimes \lpbw{\td{3}1}. 
\end{align*}

\smallbreak
\noindent 
{\it Claim: $\Delta(\lpbw{\td{3}1}^N)=\mt{a}_1^N+\mt{a}_2^N+\mt{a}_3^N+2(1+q_{33})^{-2}\mt{a}_4^N+\mt{a}_5^N$.}
\smallbreak

\noindent 
Since $\mt{a}_1\mt{a}_j=q_{33}^{-2}\mt{a}_j\mt{a}_1$ and 
$\mt{a}_2\mt{a}_k=q_{33}^{-2}\mt{a}_k\mt{a}_2$, 
for all $1\leq j\leq 5$ and $2\leq k\leq 5$, it follows that
\[
\Delta(\lpbw{\td{3}1}^N)=\mt{a}_1^N+\mt{a}_2^N+(\mt{a}_3+\mt{a}_4+\mt{a}_5)^N.
\]
Moreover, among the elements in the last summand it holds 
\begin{align*}
	\mt{a}_3\mt{a}_4= q_{33}^{-2}\mt{a}_4\mt{a}_3,\qquad 	
	\mt{a}_3\mt{a}_5= q_{33}^{-2}\mt{a}_5\mt{a}_3-\beta\mt{a}^2_4,
	\qquad \mt{a}_4\mt{a}_5= q_{33}^{-2}\mt{a}_5\mt{a}_4,
\end{align*}
where $\beta$ is given in \eqref{notation-beta}. 
Applying Proposition \ref{prop-useful} for 
$z_i=\mt{a}_{i+2}$, $i\in \I_3$, and $w_1=-\beta\mt{a}^2_4$ 
we get 
\[
(\mt{a}_3+\mt{a}_4+\mt{a}_5)^N=(\mt{a}_3+\beta_1\mt{a}_4)^N +
(\beta_2\mt{a}_4+\mt{a}_5)^N=\mt{a}^N_3+(\beta_1^N+\beta^N_2)\mt{a}^N_4+\mt{a}^N_5, 
\]
and consequently the claim follows 
because $\beta_1^N+\beta^N_2=2(1+q_{33})^{-2}$.

\smallbreak
From the previous claim and Remark \ref{obs-repe-veces}, 
taking $c_1=q_{1,\td{3}2}^{\frac{N(N-1)}{2}}$, $c_2=q_{21,3}^{N(N-1)}$ and 
$c_3=q_{31,3}^{\frac{N(N-1)}{2}}$, we have
\begin{align*}
\Delta(\lpbw{\td{3}1}^N)&=\lpbw{\td{3}1}^N\otimes 1
+c_1\xi_2^N\lpbw{\td{3}2}^Ng_1^N\otimes \lpbw{1}^N+
c_2\xi_2^N\xi_1^N\lpbw{3}^{2N}g_{21}^N\otimes\lpbw{21}^N+\\
&\quad + 2c_3(1+q_{33})^{-2}\xi_2^N \lpbw{3}^Ng_{31}^N\otimes \lpbw{31}^N  
+ g_{\td{3}1}^N\otimes \lpbw{\td{3}1}^N\\[.2em]
&=\lpbw{\td{3}1}^N\otimes 1+c_1\xi_2^N\Big(\mu_{\td{3}2}(g_{\td{3}2}^N-1)
-2\xi_1^N\mu_{3}\lpbw{32}^N-\xi_1^N\xi_2^N\mu_{3}^2\lpbw{2}^N\Big)g_1^N\otimes \lpbw{1}^N+ \\           
&\quad + c_2\xi_2^N\xi_1^N\mu_3^{2N}(g_3^N-1)^{2N} g_{21}^N\otimes\lpbw{21}^N+
2c_3(1+q_{33})^{-2}\xi_2^N \mu_3^{N}(g_3^N-1)^{N}  g_{31}^N\otimes \lpbw{31}^N +\\
&\quad + g_{\td{3}1}^N\otimes \lpbw{\td{3}1}^N.
\end{align*}
Hence $	\nu_{\td{3}1}=\lpbw{\td{3}1}^N+2\xi_1^N\mu_{3}\lpbw{31}^N+
\xi_1^N\xi_2^N\mu_{3}^2\lpbw{21}^N+
\xi_2^N\mu_{\td{3}2}\lpbw{1}^N\in \mathcal{P}_{g^N_{\td{3}1}}^{\chi_{\td{3}1}^N}$.
By Lemma \ref{lem-for-lifting} (ii), 
there is $\mu_{\td{3}1}\in \Bbbk$ such that 
$\nu_{\td{3}1}=\mu_{\td{3}1}(g_{\td{3}1}^N-1)$ and consequently
\[
\lpbw{\td{3}1}^N=\mu_{\td{3}1}(g_{\td{3}1}^N-1)-2\xi_1^N\mu_{3}\lpbw{31}^N-
\xi_1^N\xi_2^N\mu_{3}^2\lpbw{21}^N-\xi_2^N\mu_{\td{3}2}\lpbw{1}^N,
\] 
where $\mu_{\td{3}1}=0$ $g_{\td{3}1}^N=1$ or 
$\chi_{\td{3}1}^N\neq\varepsilon$.
\end{proof}

\subsection{Deformation of $N$-th power of the root vector of degree 5}\label{grade-5}
In this subsection we complete the list of all possible deformations by giving 
the lifting of the relation given by the power of the root vector of degree 5,
which is given by the element $\pbw{\alpha_{\td{2}1}}=\pbw{2}\pbw{\td{3}1}-q_{2,\td{3}1}\pbw{\td{3}1}\pbw{2}$. In order to do this, 
we define the elements:
\begin{align}\label{b-elements}
	b_2=\lpbw{32}^2 - q_{32}^{-1} q_{33}^{-2}(1+q_{33}^{-1})\lpbw{\td{3}2}\lpbw{2},
	\quad b_3=q_{33}^{-1}\lpbw{\td{3}2}-\xi_2\lpbw{3}\lpbw{32}+\xi_1\xi_2\lpbw{3}^2\lpbw{2}.
\end{align}

Following the steps of the previous subsections, we decompose 
the coproduct of $x_{\alpha_{\td{2}1}}$ as a sums of elementary 
tensors 
$\Delta(\lpbw{\td{2}1})= \mt{a}_1+\mt{a}_2+\mt{a}_3+\mt{a}_4+\mt{a}_5+\mt{a}_6$
with 
\begin{equation}\label{def-as}
\begin{aligned}
	\mt{a}_1&=\lpbw{\td{2}1}\otimes 1, & \mt{a}_4&=
	 -\xi_2  q_{33}q_{32}^{-1}\big(\lpbw{32}
	 -\xi_2 \lpbw{3}\lpbw{2}\big)    g_{31}\otimes\lpbw{31},\\
	\mt{a}_2&=-\xi_1\xi_2q_{32}^{-1}b_2g_1\otimes \lpbw{1}, &\mt{a}_5&
	= \xi_2  \lpbw{2}g_{\td{3}1}\otimes\lpbw{\td{3}1}, \\
	\mt{a}_3&=\xi_2 q_{32}^{-2}b_3 g_{21}\otimes \lpbw{21}, & \mt{a}_6&
	= g_{\td{2}1}\otimes \lpbw{\td{2}1}.
\end{aligned}
\end{equation}

\begin{lema}\label{def-grade-5} 
$
\Delta(\lpbw{\td{2}1}^N)= \mt{a}^N_1+\mt{a}^N_2+\mt{a}^N_3+2(1+q_{33})^{-N}\mt{a}^N_4+\mt{a}_5^N   +\mt{a}_6^N.
$
\end{lema}

\begin{proof}
From the fact $\mt{a}_1\mt{a}_i=\xi_2\mt{a}_i\mt{a}_1$ for all $i\in \I_6$, 
we obtain 
\begin{align}\label{eq-reduction-1}
\big(\sum_{i=1}^{6}\mt{a}_i\big)^N=\mt{a}_1^N+\big(\sum_{i=2}^{6}\mt{a}_i\big)^N. 
\end{align}
Let $w_1=\mt{a}_2\mt{a}_6-\mt{a}_6\mt{a}_2$ and $w_2=\gamma\mt{a}_4^2- w_1$, with $\gamma= (1-q_{33}^{-1})/(1+q_{33})$. Then we have
\begin{align*}
\mt{a}_2\mt{a}_6&=\xi_2\mt{a}_6\mt{a}_2+w_1, & 	\mt{a}_3\mt{a}_5&=\xi_2\mt{a}_5\mt{a}_3+ w_2, & \mt{a}_i\mt{a}_j&=\xi_2\mt{a}_j\mt{a}_i,\,\,3\leq i<j\leq 6,\,\,\,i+j\neq 8.
\end{align*}
Since $\gamma=\gamma_1\gamma_2(1-q_{33}^{-2})$ with  $\gamma_1=1/(1+q_{33})$ and $\gamma_2=q_{33}/(1+q_{33})$, 
the result follows from  Proposition \ref{prop-useful}.
\end{proof}

The next lemma gives us the expression of $\mt{a}^N_2$.

\begin{lema} \label{lem-comutation-recur}
Let $i,j,k,n\geq 1$.
\begin{enumerate}
	\item [\rm (i)]  The following identities hold:
	\begin{align*}
		\lpbw{2}\lpbw{3}^i&= q_{32}^{-i}\big(\lpbw{3}^i\lpbw{2} - 	
		(i)_{q_{33}^{-1}}\lpbw{3}^{i-1}\lpbw{32}+
		q_{33}^{-1}\tbinom{i}{2}_{q_{33}^{-1}}\lpbw{3}^{i-2}\lpbw{\td{3}2}\big),\\[.2em]
		\lpbw{32}\lpbw{3}^i&= q_{33}^{-i} q_{32}^{-i}\big(\lpbw{3}^i \lpbw{32} - 
		  (i)_{q_{33}^{-1}} \lpbw{3}^{i-1}\lpbw{\td{3}2} \big),\\[.2em]
		\lpbw{\td{3}2}\lpbw{3}^i&=q_{33}^{-2i} q_{32}^{-i}\lpbw{3}^i\lpbw{\td{3}2},\\[.2em]
		\lpbw{2} \lpbw{\td{3}2}^j&=q_{33}^{-2j} q_{32}^{-2j}\big( \lpbw{\td{3}2}^j \lpbw{2}
		  -q_{33}^{2} q_{32}\xi_1  (j)_{q_{33}^{-2}}\lpbw{\td{3}2}^{j-1}\lpbw{32}^2\big), \\[.2em]
		\lpbw{32}\lpbw{\td{3}2}^j&= q_{33}^{-2j}q_{32}^{-j} \lpbw{\td{3}2}^j \lpbw{32},\\[.2em]
		\lpbw{2}\lpbw{32}^k&= q_{33}^{-2k} q_{32}^{-k}\lpbw{32}^k\lpbw{2}.
	\end{align*}
	\item [\rm (ii)] If $b_2$ is the element given in \eqref{b-elements} then
	\[b_2^n=\!\!\!\sum_{
		j+k=n}\!\!\!  (-1)^jq_{32}^{-j(2n-j)}    
	\binom{n}{j,k}_{q_{33}^{-2}} q_{33}^{-j(2n-j+1)} (1+q_{33}^{-1})^{j}\lpbw{\td{3}2}^j \lpbw{32}^{2k} \lpbw{2}^j.\]
	In particular, $b_2^N=\lpbw{32}^{2N}-(1+q_{33}^{-1})^N \lpbw{\td{3}2}^N \lpbw{2}^N$.
	
	 	\item [\rm (iii)]  The following identity hold:
	 	\[
		\mt{a}^N_2=-\xi_1^N\xi_2^N \big(\lpbw{32}^{2N}
		-(1+q_{33}^{-1})^N \lpbw{\td{3}2}^N \lpbw{2}^N\big) g_1^N\otimes \lpbw{1}^N.
		\]	 
\end{enumerate}
\end{lema}

\begin{proof}
All identities of (i) are proved by induction and their proofs will be omitted. 
Using (i) and induction on $n$ we obtain the first part of (ii). 
The last part of (ii) follows from the fact that $\binom{N}{j,k}_{q_{33}^{-2}}\neq 0$
 only in the cases $(j,k)=(0,N)$ and $(j,k)=(N,0)$. 
 For (iii), observe that the last part of (ii) and 
 Remark \ref{obs-repe-veces} imply that
\begin{align*}
	\mt{a}^N_2=-\xi_1^N\xi_2^N q_{32}^{-N} q_{1,32}^{-N(N+1)}\big(\lpbw{32}^{2N}-
	q_{32}^{-N^2}(1+q_{33}^{-1})^N \lpbw{\td{3}2}^N \lpbw{2}^N\big) g_1^N\otimes \lpbw{1}^N.	
\end{align*}
\end{proof}

Now, we compute $a^N_3$.

\begin{lema}\label{lem-recu-2}
	Let $r,s,t\geq 0$,  $n\geq 1$ and  $\zeta_{r,s,t}=q_{33}^{-r}(1-q_{33}^{-1})^s(1+q_{33}^{-1})^t$. 
\begin{enumerate}
	\item [\rm (i)] If $b_3$ is the element given in \eqref{b-elements} then 
\[b_3^n=\!\!\!\sum_{\substack{j+k+l=n\\
		i+2j+k=2n}}\!\!\! (-1)^i\, \zeta_{j,i,k+l}\, q_{32}^{w(j,l)} 
		 \tbinom{n}{j,k,l}_{q_{33}^{-2}} \, \tfrac{(k)_{q_{33}^{-2}}!}{(k)_{q_{33}^{-1}}!}\,\,
\lpbw{3}^i\lpbw{\td{3}2}^j\lpbw{32}^k\lpbw{2}^l,\]
where $w(j,l)= -2\tbinom{n}{2}+\tbinom{j}{2}+\tbinom{n-l}{2}$ and $\tbinom{r}{s}=0$ if $r<s$. 
In particular, 
\begin{align*}
	b_3^N=\lpbw{\td{3}2}^N-  2(1+q_{33}^{-1})^N   \xi_2^N \lpbw{3}^N\lpbw{32}^N+
	\xi_1^N\xi_2^N\lpbw{3}^{2N}\lpbw{2}^N.	
\end{align*}

\item [\rm (ii)] The following identity hold:
\begin{align*}
		\mt{a}^N_3=	\xi_2^N \, b_3^N g_{21}^N\otimes \lpbw{21}^N.
\end{align*}
\end{enumerate}
\end{lema}

\begin{proof}
(i) The first part follows by induction on $n$ and the proof will be omitted.  For the second part, 
notice that $\tbinom{n}{j,k,l}_{q_{33}^{-2}}\neq 0$ only if either 
$(i,j,k,l)=(0,N,0,0)$, or  $(i,j,k,l)=(N,0,N,0)$ or $(i,j,k,l)=(2N,0,0,N)$. 
Also, for $k=N$, we have
\[ 
\frac{(k)_{q_{33}^{-2}}!}{(k)_{q_{33}^{-1}}!}= 
\frac{(N)_{q_{33}^{-2}}!}{(N)_{q_{33}^{-1}}!}=
 \frac{\prod_{1\le i\le N}(1+q_{33}^{-i})}{(1+q_{33}^{-1})^N}\overset{(\ast)}{=}\frac{2}{(1+q_{33}^{-1})^N},
\]
where the equality $(\ast)$ follows from 
$t^N-1=\prod_{1\le i\le N}(t-q_{33}^{-i})$ when $t=-1$. 
From these previous facts, the second part of (i) follows. \smallbreak

\noindent (ii) From Remark \ref{obs-repe-veces} it follows that
%
	$$\mt{a}^N_3=(\xi_2 q_{32}^{-2}b_3 g_{21}\otimes \lpbw{21})^N
	=\xi_2^N q_{32}^{-2N}q_{21,\td{3}2}^{\tfrac{N(N-1)}{2}}b_3^N g_{21}^N\otimes \lpbw{21}^N
	=\xi_2^N b_3^N g_{21}^N\otimes \lpbw{21}^N.$$

\end{proof}

Finally, we compute $a^N_4$.

\begin{lema}\label{lem-recu-3}
	Let $r,s,t\geq 0$ and  $\zeta_{r,s,t}=q_{33}^{-r}(1-q_{33}^{-1})^s(1+q_{33}^{-1})^t$.
\begin{enumerate}
	\item [\rm (i)]   If $n\geq 1$, then 
	
	\[(\lpbw{32}-\xi_2\lpbw{3}\lpbw{2})^n=\!\!\!\sum_{\substack{j+k+l=n\\
			i+2j+k=n}}\!\!\!  (-1)^i q_{32}^{\td{w}(j,l)}\, \tfrac{(j)_{q_{33}^{-1}}! 
			(l)_{q_{33}^{-1}}!}{{(j)_{q_{33}^{-2}}!\,(i)_{q_{33}^{-1}}!}}\, 
			 \tbinom{n}{j,k,l}_{q_{33}^{-1}} \zeta_{j,i+j,i}	\,\, 
			   \lpbw{3}^i \lpbw{\td{3}2}^j \lpbw{32}^k \lpbw{2}^l\]
	where $\td{w}(j,l)= -\tbinom{n}{2}+\tbinom{j}{2}+\tbinom{n-l}{2}$ and
	 $\tbinom{r}{s}=0$ if $r<s$. In particular, 	
	\begin{align*}
		(\lpbw{32}-\xi_2\lpbw{3}\lpbw{2})^N=\lpbw{32}^N-\xi_2^N  \lpbw{3}^N\lpbw{2}^N.
	\end{align*}
	
	\item [\rm (ii)] The following identity hold:
		\begin{align*}
		\mt{a}^N_4=	-\xi_2^N \,\big( \lpbw{32}^N-\xi_2^N 
		 \lpbw{3}^N\lpbw{2}^N \big)g_{31}^N\otimes \lpbw{31}^N.
	\end{align*}
\end{enumerate}
\end{lema}	

\begin{proof}
The first part of (ii) follows by induction on $n$. 
The second part of (i) is a consequence from the fact that 
$\tbinom{n}{j,k,l}_{q_{33}^{-1}}\neq 0$ only if either 
$(i,j,k,l)=(0,0,N,0)$ or $(i,j,k,l)=(N,0,0,N)$. In order to prove (ii), 
notice that by Remark \ref{obs-repe-veces},
		\begin{align*}
			\mt{a}_4^N&= \big(-\xi_2  q_{33}q_{32}^{-1}\big(\lpbw{32}
			-\xi_2 \lpbw{3}\lpbw{2}\big)  
			  g_{31}\otimes\lpbw{31}\big)^N\\[.2em]
			&= -\xi_2^N q_{32}^{-N}  q_{31,32}^{\tfrac{N(N-1)}{2}}\big(\lpbw{32}
			-\xi_2 \lpbw{3}\lpbw{2}\big)^N g_{31}^N\otimes\lpbw{31}^N\\[.2em]
			&=-\xi_2^N q_{32}^{-N}  q_{31,32}^{\tfrac{N(N-1)}{2}}\,
			\big( \lpbw{32}^N-\xi_2^N q_{32}^{-\tfrac{N(N - 1)}{2}}
			 \lpbw{3}^N\lpbw{2}^N \big)g_{31}^N\otimes \lpbw{31}^N,
		\end{align*}
and the result is proved.
	\end{proof}

In order to prove that main result of this subsection, 
we introduce the following scalars 
\begin{equation}\label{scalars-alpha}
	\begin{aligned}
		\lambda_{\td{3}1}&=\xi_2^N\mu_2,&
		\lambda_{31}&= 2\xi_1^N(\xi_2^N\mu_2\mu_3-\mu_{32}),\\	
		\lambda_{21}&=\xi_2^N(\xi_1^N\xi_2^N\mu_2\mu_3^2-2\xi_1^N\mu_3\mu_{32}+\mu_{\td{3}2}),&
		\lambda_{1}&=\xi_2^N(\xi_2^N\mu_2\mu_{\td{3}2}-\xi_1^N\mu_{32}^2).
	\end{aligned}
\end{equation} 

In the next lemma we present explicitly the deformation in degree 5.

\begin{lema}\label{prop-def-23321}
Let $\lambda_{\td{3}1}$, $	\lambda_{31}$, 	$\lambda_{21}$ 
and $\lambda_{1}$ as in \eqref{scalars-alpha}. Then there exists $\mu_{\td{2}1}\in \Bbbk$ such that 
	\begin{align}\label{def-rel-grade-3=y23321}
		\lpbw{\td{2}1}^N=\mu_{\td{2}1}(g_{\td{2}1}^N-1)
		-\lambda_{\td{3}1}\lpbw{\td{3}1}^N-\lambda_{31}\lpbw{31}^N
		-\lambda_{21}\lpbw{21}^N-\lambda_{1}\lpbw{1}^N,
	\end{align}
	and $\mu_{\td{2}1}=0$ if $g^N_{\td{2}1}=1$ or $\chi^N_{\td{2}1}\neq \epsilon$.
\end{lema}

\begin{proof}
From Lemma \ref{def-grade-5}, Lemma \ref{lem-comutation-recur} (iii), 
Lemma \ref{lem-recu-2} (ii) and Lemma \ref{lem-recu-3} (ii) it follows
\[
	\nu_{\td{2}1}= \lpbw{\td{2}1}^N+\alpha_{\td{3}1}\lpbw{\td{3}1}^N+
	\alpha_{31}\lpbw{31}^N+\alpha_{21}\lpbw{21}^N+\alpha_{1}\lpbw{1}^N\in 
	\mathcal{P}_{g^N_{\td{2}1}}^{\chi_{\td{2}1}^N}(A).
\]
Then, by Lemma \ref{lem-for-lifting} (ii), there is $\mu_{\td{2}1}\in \Bbbk$ 
such that $\nu_{\td{2}1}=\mu_{\td{2}1}(g_{\td{2}1}^N-1)$ 
and the assertion is proved.
\end{proof}


\section{Liftings of Nichols algebras of type $B_3$}
Let $\Gamma$, $V$ and $B=\toba(V)\#\Bbbk\Gamma$ be as in the paragraph
just before Section \ref{section-deform}. In particular, we keep
the notation for the elements $g_{\alpha} \in \Gamma$ and $\chi_{\alpha} 
\in \widehat{\Gamma}$ as in \eqref{notation-g}, \eqref{notation-chi} and 
\eqref{notation-chi-alpha}.

\subsection{Presentation of the algebras}
In order to be as clear and precise as possible,
we split the data for the presentation of the algebras into 
the following subsections.

\subsubsection*{The generators}
Let $R$ be the free associative algebra generated by 
elements $y_{\alpha}$ with $\alpha\in \Phi^{+}$ 
and $y_{i}:=y_{\alpha_{i}}$ for $i\in \I_{3}$, 
satisfying the relations given in \eqref{root-vector}, 
and the quantum Serre relations  
\[
\ad(y_{i})^{1-a_{ij}}({y_{j}})=0,\qquad i\neq j.
\]
The explicit description of all commutation 
relations arising from the previous ones,
can be seen in the Remark \ref{rmk:PBW-relations}.
Using the same arguments of \cite[p.11]{Bea}, one can check
that $R$ is a Hopf algebra in the category of Yetter-Drinfeld modules 
over $\Bbbk \Gamma$, where the structure is determined by setting 
$g\cdot y_{i} = \chi_{i}(g) y_{i}$ and $\delta(y_{i})= g_{i}\ot y_{i}$,
for all $1\leq i\leq 3$ and $g\in \Gamma$.
In particular, 
the bosonization $R\# \Bbbk \Gamma$ is an ordinary Hopf algebra. 

\subsubsection*{The parameters}
Let $(\mu_{\alpha})_{\alpha \in \Phi^{+}}$ be a family of scalars satisfying 
that $\mu_{\alpha} = 0$ if $g_{\alpha}^{N} = 1$ or $\chi_{\alpha}^{N}\neq \varepsilon$. Take $\xi_1, \xi_2$ as in \eqref{xi-scalar}.

\subsubsection*{The ideal of relations $J$}
Consider the two-sided ideal $J$ of $R\#\Bbbk\Gamma$ generated by the elements
corresponding to the relations in \eqref{def-rel-grade-1}, 
\eqref{def-rel-grade-2}, \eqref{def-rel-grade-3=y321}, 
\eqref{def-rel-grade-3=y332}, \eqref{def-rel-grade-3=y3321} and 
\eqref{def-rel-grade-3=y23321}.

\subsubsection*{The presentation}
For $\Gamma$, $V$, $\mathbf{\mu}=(\mu_{\alpha})_{\alpha \in \Phi^{+}}$ and
$\xi_{1}$, $\xi_{2}$ 
as above, 
we define the algebra $A(\Gamma, V, \mathbf{\mu})$ by the quotient 
$R\#\Bbbk\Gamma /J$. Explicitly, it is the algebra generated by the 
elements $y_{\alpha}$, $\alpha \in \Phi^{+}$, the group $\Gamma$ and satisfying the following relations:

\begin{align*}
gy_{i}g^{-1} &= \chi_{i}(g)y_{i}, 
\qquad \text{ for all }i \in \I_{3}, \,g\in \Gamma ,\\[.5em]
\ad(y_{i})^{1-a_{ij}}(y_{j})&= 0,\qquad 
\text{ for all } i\neq j,\ i,j\in \I_{3},\\[.5em]
 y_{i}^N& = \mu_i(g_i^N-1),\qquad   \text{ for all } i\in \I_3,\\[.5em]
y_{(i+1)i}^N & =\mu_{(i+1)i}\big(g_{(i+1)i}^N-1\big)
-\xi_2^N\mu_{i+1}\,\mu_{i}(g_i^N-1),\qquad \text{ for all } i\in \I_2,\\[.5em]
y_{31}^N & =\mu_{31}(g_{31}^N-1)
-\xi_2^N\mu_{3}\,\mu_{21}\big(g_{21}^N-1\big)-\xi_2^N(\mu_{32}-\xi_2^N\mu_{3}\mu_{2})\,
\mu_1(g_1^N-1),\\[.5em]
y_{\td{3}2}^N&=\mu_{\td{3}2}(g_{\td{3}2}^N-1)-2\xi_1^N\mu_{3}\mu_{32}\,(g_{32}^N-1)+\xi_1^N\xi_2^N\mu_{3}^2\mu_{2}\,(g_{2}^N-1),\\[.5em]
y_{\td{3}1}^N & = \mu_{\td{3}1}(g_{\td{3}1}^N-1)-2\xi_1^N\mu_{3}\mu_{31}\, (g_{31}^N-1)+\xi_1^N\xi_2^N\mu_{3}^2\mu_{21}\,(g_{21}^N-1)\\
&\quad -\xi_2^N\,(\xi_1^N\xi_2^N\mu_{3}^2\mu_{2}-2\xi_1^N\mu_{3}\mu_{32}+\mu_{\td{3}2})\mu_{1}\,(g_{1}^N-1),\\[.5em]
y_{\td{2}1}^N&= 
\mu_{\td{2}1}(g_{\td{2}1}^N-1)
-\xi_2^N\mu_{2}\mu_{\td{3}1}\, (g_{\td{3}1}^N-1)
+2\xi_1^N\mu_{32}\mu_{31}\,(g_{31}^N-1)\\
&\quad -\xi_2^N\, \mu_{\td{3}2}\mu_{21}\,(g_{21}^N-1)
+\xi_2^N\,(
\xi_2^N\mu_{2}\mu_{\td{3}2}
-\xi_1^N\mu_{32}^2
)\mu_{1}\,(g_{1}^N-1).
\end{align*}	

\bigbreak
\begin{prop}\label{prop:A-hopfalg}
$A(\Gamma, V, \mathbf{\mu})$ is a pointed Hopf algebra whose structure
is determined by 
$$
\com(g) = g\ot g, \qquad 
\com(y_{i}) = y_{i}\ot 1 + g_{i}\ot y_{i}\qquad\text{ for all }g\in \Gamma,
i\in\I_{3}.
$$
\end{prop}

\begin{proof}
It is enough to prove that the two-sided ideal $J$ of $R\#\Bbbk\Gamma$
is actually a Hopf ideal. But this follows immediately from the proof of 
Lemma \ref{lem:def-degree-1}, Lemma \ref{lem:def-degree-2}, 
Lemma \ref{lem:def-degree-3}, Lemma \ref{lem-grade-4} and 
Lemma \ref{prop-def-23321}, for all the elements corresponding to
relations \eqref{def-rel-grade-1}, 
\eqref{def-rel-grade-2}, \eqref{def-rel-grade-3=y321}, 
\eqref{def-rel-grade-3=y332}, \eqref{def-rel-grade-3=y3321} and 
\eqref{def-rel-grade-3=y23321} 
are skew-primitives.
Finally, the Hopf algebra is clearly pointed because it is generated by group-like 
and skew-primitive elements.
\end{proof}

%

%


\subsection{Liftings}
We finish the paper with 
the next result, which is our main contribution. 
It describes explicitly the liftings of Nichols algebras of type $B_3$.
Its proof follows from what we have done until now. Nevertheless, we add 
some lines for completeness.

\begin{teor} \label{lifting-b3} The following assertions hold:
\begin{enumerate}	
	\item [\rm (i)]  The Hopf algebra 
	$A(\Gamma, V, \mu):=R\#\Bbbk\Gamma/J$ is a lifting of $\toba(V)\#\Bbbk\Gamma$.
	In particular, $\dim A(\Gamma, V, \mu) =N^{9}|\Gamma|$.
	
	\item [\rm (ii)] If $A$ is a pointed Hopf algebra such that 
	$\operatorname{gr} A\simeq \toba(V)\#\Bbbk\Gamma$ and $\ord q= N$ is an 
	odd number greater than $5$ then $A\simeq A(\Gamma, V, \mu)$ 
	for some family of scalars $\mu$ and for some $g_i\in \Gamma$.
\end{enumerate}	
\end{teor}

\begin{proof} (i) Let $V$ be the vector space spanned by the elements 
 $y_{1}$, $y_{2}$ and $y_{3}$ and consider the Hopf algebra $\mathcal{T}(V) = T(V)\# \Bbbk \Gamma$
given by the bosonization, which is constructed using the $\Gamma$-module and $\Gamma$-comodule 
defined in the first subsection. Define the algebra map $\phi: \mathcal{T}(V) \to A(\Gamma, V, \mu)$
by $\phi(y_{i}) = y_{i}$ for all $i\in\I_{3}$ and $\phi(g) = g$ for all $g\in \Gamma$.
Clearly, $\phi$ is a Hopf algebra map and $\phi|_{\Bbbk \Gamma} = \id$.  Moreover,
we have that 
$\phi|_{V\#\Bbbk \Gamma} : V\#\Bbbk \Gamma \to \mathcal{P}_{1}(A(\Gamma, V, \mu)) $
is a Hopf bimodule isomorphism, because 
$\mathcal{P}_{1}(A(\Gamma, V, \mu))= V\#\Bbbk\Gamma$, by the Taft-Wilson Theorem and Lemma \ref{lem-g-primitivos}.
Hence, $\phi$ is a lifting map (c.f. \cite[Definition 4.3]{AAGMV}) and 
whence $A(\Gamma, V, \mu)$ is a lifting of $\toba(V)\#\Bbbk\Gamma$. 

(ii) Let 
$A$ be a pointed Hopf algebra such that 
$\operatorname{gr} A\simeq \toba(V)\#\Bbbk\Gamma$.
Then, by  
Lemma \ref{lem:def-degree-1}, Lemma \ref{lem:def-degree-2}, 
Lemma \ref{lem:def-degree-3}, Lemma \ref{lem-grade-4} and 
Lemma \ref{prop-def-23321},
there exists a Hopf algebra epimorphism $\pi:A(\Gamma, V, \mu) \to A $
for some family of scalars satisfying 
that $\mu_{\alpha} = 0$ if $g_{\alpha}^{N} = 1$ or $\chi_{\alpha}^{N}\neq \varepsilon$.
Since both algebras have the same dimension, $\pi$ is indeed an isomorphism
and the proof follows.
\end{proof}

\section*{Acknowledgments}
Part of this work was carried out during a visit of the second author to the 
Department of Mathematics at the Federal University of Santa Catarina. G. A. G. 
expresses his gratitude for the warm hospitality received, in special to  
A. do Nascimento Oliveira Sousa and O. M. He also thanks CAPES for the financial support provided.

	\end{document}